\documentclass{amsart}

\usepackage{amssymb}

\usepackage{hyperref} 

\newtheorem{theorem}{Theorem}[section]
\newtheorem{cor}[theorem]{Corollary}
\newtheorem{lemma}[theorem]{Lemma}
\newtheorem{proposition}[theorem]{Proposition}

\newtheorem{remark}[theorem]{Remark}

\newtheorem{example}[theorem]{Example}

\numberwithin{equation}{section}

\def\beq{\begin{equation}}
\def\eeq{\end{equation}}

\def\ben{\begin{enumerate} }
\def\een{\end{enumerate} }

\def\nc{non-commutative}

\def\smin{ \sigma^{min}  }

\def\cC{{\mathcal C}}

\def\cD{{\mathcal D}}
\def\cG{{\mathcal G}}
\def\cH{{\mathcal H}}

\def\cL{{\mathcal L}}

\def\cO{{\mathcal O}}
\def\cQ{{\mathcal Q}}

\def\cN{{\mathcal N}}
\def\cM{{\mathcal M}}

\def\cP{{\mathcal P}}
\def\cR{{\mathcal R}}
\def\cS{{\mathcal S}}
\def\cT{{\mathcal T}}
\def\cU{{\mathcal U}}
\def\cV{{\mathcal V}}
\def\cZ{{\mathcal Z}}

\def\matn{\mathbb R^{n\times n}}

\def\Smatn{\matn_{sym}}
\def\cW{\mathcal W}

\def\gtupn{(\mathbb{R}^{n\times n}_{sym})^g}

\def\Smatng{ (\Smatn)^g}

\def\matn{\mathbb R^{n\times n}}

\def\RRx{\RR \langle x \rangle}

\def\cK{\mathcal K}

\def\ntn{n\times n}

\newcommand{\boundary}[1]{\partial \mathcal D_{#1}}

\newcommand{\posdom}[1]{\mathcal D_{#1}}

\newcommand{\SMATN}[1]{\mathbb R_{sym}^{{#1}\times{#1}}}

\newcommand{\RR}{{\mathbb R}}

\newcommand{\SMATNG}[1]{(\SMATN{#1})^g}

\newcommand{\gh}{\mathfrak{h}}

\newcommand{\tp}[3]{\mathcal T_{#1}({#2},{#3})}
\newcommand{\csff}[4]{-\langle{#1}^{\prime\prime}(#2)[#3]#4,#4\rangle}

\newcommand{\graded}[2]{\mathcal B({#1},{#2})}

\makeindex

\begin{document}

\def\RR{\mathbb R}


\setcounter{page}{1}
\title[Non-commutative Curvature]
{Non-commutative Varieties with  Curvature
having Bounded Signature}

\author[H.~Dym, J.~W.~Helton, and
S.~McCullough]{Harry Dym, Bill Helton$^1$,
  and Scott McCullough$^2$}

\address{Department of Mathematics\\
  Weizmann Institute\\
   Rehovot, 76100 \\
   Israel}

\email{harry.dym@weizmann.ac.il}

\address{Department of Mathematics\\
   University of California, San Diego\\
   La Jolla, CA 92093-0112\\
   USA}

\email{helton@math.ucsd.edu}

\address{Department of Mathematics\\
  University of Florida\\
  Box 118105\\
  Gainesville, FL 32611-8105\\
  USA}

\email{sam@math.ufl.edu}



\thanks{
\quad ${}^1$ Research supported by National Science
 Foundation grants DMS 07758 and DMS 0757212,
and the Ford Motor Co, \quad ${}^2$ Research
  supported by the NSF}

\date{\today}

\begin{abstract}
 A  natural notion 
for the signature
 $C_{\pm}(\cV(p))$ of the curvature of the zero set $\cV(p)$ of a non-commutative polynomial $p$  is introduced. The main result of this paper is the bound
$$ \textup{deg} \, p \ \leq \ 2 C_\pm (\cV(p) ) + 2.$$
It is obtained under some  irreducibility and
 nonsingularity conditions, and shows that the signature of the curvature
of the zero set of $p$
dominates its degree.

The condition
  $C_+(\cV(p))=0$ means that the non-commutative
  variety
   $\cV(p)$ has positive curvature. In this case
the preceding inequality implies that
the degree of $p$ is at most two.
Non-commutative
varieties $\cV(p)$ with positive curvature
were introduced in \cite{DHMind}.
There a slightly weaker irreducibility hypothesis
plus a number of additional hypotheses
yielded a weaker result on $p$.
The approach here is quite different; it is 
cleaner, and allows for the
treatment of arbitrary signatures.

In  \cite{DGHM} the degree of
a \nc \ polynomial $p$ was bounded by twice the signature
of its Hessian plus two.
 In this paper we introduce
 a modified version of this non-commutative Hessian
 of $p$
 which turns out to be very appropriate for analyzing
 the  
variety $\cV(p)$.
\end{abstract}

\maketitle

\bigskip

\newpage

\section{Introduction}


In the classical setting of a surface defined by the zero set
$\nu(p)$ of a
polynomial $p=p(x)=p(x_1,\dots,x_g)$ in $g$ commuting variables,
the second fundamental form at a smooth point
$x_0$ of $\nu(p)$ is the quadratic  form,
\begin{equation}
\label{def:2ndfform}
  -\langle (\text{Hess}\,p)(x_0) h, h\rangle,
\end{equation}
 where $\text{Hess}\,p$ is the Hessian of $p$, and $h\in\RR^g$ is
 in the tangent space to the surface $p(x)=0$
 at $x_0$; i.e., $\nabla p(x_0)\cdot h=0$.
\footnote{
The choice of the minus sign in (\ref{def:2ndfform}) is somewhat arbitrary.
Classically the sign of the second fundamental form is associated with the
choice of a smoothly varying vector that is normal to $\nu(p)$.
The zero set $\nu(p)$ has positive curvature
 at $x_0$ if the second fundamental form is
either positive semidefinite or  negative semi-definite at $x_0$.
For example,
if we define $\nu(p)$ using a concave function $p$,
then the second fundamental form is
negative semidefinite, while
for the same set $\nu(-p)$ the second fundamental form is
positive  semidefinite.
}

 In this paper we show that in the non-commutative setting
 even a modicum of positive curvature of
 the zero set $\cV(p)$ for a non-commutative
 polynomial $p$ (subject to appropriate irreducibility constraints)
implies that $p$ is convex - and thus, $p$ 
 has degree at most two - and $\cV(p)$ has
 positive curvature everywhere; see
 Theorem \ref{thm:sigmain}
 and its corollary, Corollary \ref{cor:sigmain},
 for the precise statements.
In addition, we introduce
  a natural
  notion of the signature $C_\pm(\cV(p))$ of a variety $\cV(p)$ 
   and obtain the bound 
$$ \textup{deg} \, p \ \leq \ 2 C_\pm (\cV(p) ) + 2$$
on the degree of $p$ in terms of the signature $C_\pm (\cV(p))$.


  Throughout the paper we shall adopt
  the convention that
    $C_+(\cV(p))=0$  corresponds to positive curvature,
 since in our examples, defining functions $p$ are
typically concave or  quasiconcave.
The convention $C_-(\cV(p))=0$ or even
$C_\pm(\cV(p))=0$ would be equally reasonable.
Either way gives the same mathematical consequences.
%

Now that the main results have been described informally,
the remainder of this introduction turns to the precise statements.
 The setting of this paper
 overlaps   that of \cite{DHMind}.
The principal
 definitions are reviewed briefly for the convenience of the reader
 in  Subsections \ref{subsec:polys}, \ref{subsec:previous},
  \ref{subsec:smooth}, \ref{subsec:alg-open},
 \ref{subsec:ncsag}, \ref{subsec:irreducible};
  the notion of positive curvature is introduced in Subsection
  \ref{subsec:curvature};
  and
 the main new results are stated in Subsections
 \ref{subsec:results1} and \ref{subsec:results2}. The introduction
 concludes with a guide to the rest of the paper in Subsection
\ref{subsec:guide}.

\subsection{NC Polynomials}
 \label{subsec:polys}
Let $x=\{x_1,\dots,x_g\}$  denote non-commuting indeterminates
and let $\RRx$
denote the set of  polynomials $$p(x)=p(x_1,\ldots,x_g)$$
\index{polynomial, nc}
in the indeterminates $x$ with real coefficients; i.e., the set of finite
linear combinations
\begin{equation}
\label{eq:sep6a10}
  p=\sum_{|w|\leq d} c_w w \quad\textrm{with}\quad c_w\in \mathbb R
\end{equation}
of words $w$ in $x$. The degree of such a polynomial $p$
is defined
as the maximum of the lengths $\vert w\vert$ of
the words $w$ appearing (nontrivially) in the linear combination
(\ref{eq:sep6a10}). Thus,
for example, if $g=3$, then
$$
p_1= 3x_1 x_2^3 + x_2 + x_3x_1x_2 \quad\textrm{and}\quad
p_2= 2x_1 x_2^3 + x_2^3 x_1 + x_3x_1x_2 +x_2 x_1 x_3
$$
are polynomials of degree four in $\RRx$.

There is a natural
 {\bf involution} \index{involution} $w^T$ on words given by the rule
$$
x_j^T=x_j\quad\textrm{and if}\quad w=x_{i_1}x_{i_2}\cdots x_{i_k},\quad
\textrm{then}\quad
w^T=x_{i_k}\cdots x_{i_2}x_{i_1},
$$
which extends to polynomials $p=\sum c_w w$ by linearity:
\begin{equation*}
 p^T =\sum_{|w|\le d} c_w w^T.
\end{equation*}

A polynomial $p\in\RRx$ is said to be {\bf symmetric} \index{symmetric}
if $p=p^T$. The
second polynomial $p_2$ listed above is symmetric, the first is not.
Because of the requirement $x_j^T=x_j$, the variables
are said to be symmetric.

A polynomial $p(x)=p(x_1,\ldots,x_g)$ in non-commuting variables
$\{x_1,\ldots,x_g\}$ will be referred to as an {\bf nc polynomial}
\index{nc polynomial} for short; and
{\bf nc} \index{nc} will be used as a short hand notation for non-commutative.

\subsubsection{Substituting Matrices for Indeterminates}
\label{sec:openG2}

Let $\gtupn$ denote the set of $g$-tuples
$(X_1,\ldots,X_g)$ of real symmetric $n\times n$ matrices.
We shall be interested in evaluating a
polynomial $p(x)=p(x_1,\ldots,x_g)$
that belongs to $\RRx$ at a
tuple $X=(X_1,\dots,X_g)\in(\mathbb{R}^{n\times n}_{sym})^g$.
In this
case $p(X)$ is also an $n\times n$ matrix and
the involution on $\RRx$
that was introduced earlier is compatible
with matrix transposition, i.e.,
$$
p^T(X)=p(X)^T,
$$
where $p(X)^T$ denotes the transpose of the  matrix  $ p(X)$.
When $X\in \gtupn$ is substituted into $p$
the constant term  $p(0)$ of $p(x)$  becomes $p(0) I_n$.
For example, if $p(x)=3+x^2$, then
$$ p(X)= 3 I_n+X^2.$$

\indent A symmetric polynomial $p \in \RRx$ is said to be
{\bf matrix positive}
if $p(X)$ is a positive semi-definite matrix for each tuple
$X=(X_1,\dots,X_g)\in (\mathbb{R}^{n\times n}_{sym})^g$.
Similarly,  $p$  is said to be
{\bf matrix convex} \index{matrix convex} if
\begin{equation}
\label{eq:nov13b6}
p(tX+(1-t)Y)\preceq tp(X)+(1-t)p(Y)
\end{equation}
for every pair of tuples $X, Y\in(\mathbb{R}^{n\times n}_{sym})^g$ and
$0\le t\le 1$.

\subsubsection{Derivatives}
We define the directional
derivative \index{derivative} of the
word $w=x_{j_1}x_{j_2}\cdots x_{j_n}$ with coefficient $c\in\RR$ 
as the linear form:
\begin{equation*}
w^\prime[h] = h_{j_1}x_{j_2}\cdots x_{j_n}
       + x_{j_1}h_{j_2}x_{j_3}\cdots x_{j_n} + \ \cdots \
       + x_{j_1}\cdots x_{j_{n-1}}h_{j_n}
\end{equation*}
and extend the definition to polynomials $p = \sum c_w w$ by linearity; i.e.,
\begin{equation*}
p^\prime(x)[h]=\sum c_w w'[h].
\end{equation*}

Thus, $p^\prime(x)[h] \in \mathbb R\langle x,h \rangle$ is
the
$$
\textrm{coefficient of $t$ in the expression}\quad p(x+th)-p(x);
$$
it is an nc
polynomial in $2g$ (symmetric)
variables $(x_1,\dots,x_g,h_1,\dots,h_g)$.
Higher order derivatives are computed by the same recipe, i.e., as the coefficient of $t$ in the 
expression $q(x+th)[h]-q(x)[h]$: If
$q(x)[h]=h_{j_1}x_{j_2}\cdots x_{j_n}$, then
$$
q^{\prime}(x)[h]=h_{j_1}h_{j_2}x_{j_3}\cdots x_{j_n}
+h_{j_1}x_{j_2}h_{j_3}\cdots x_{j_n}+\cdots+h_{j_1}x_{j_2}\cdots
x_{j_{n-1}}h_{j_n}
$$
and the definition is extended to finite linear combinations of such terms
by linearity.
If $p$ is symmetric, then so is $p^\prime$.
For $g$-tuples of symmetric matrices of a fixed size
$X,H ,$  the evaluation formula
\begin{equation*}
p^\prime(X)[H]=\lim_{t\to 0} \frac{p(X+tH)-p(X)}{t}
\end{equation*}
holds, and  if $q(t)=p(X+tH)$, then
\begin{equation}
\label{eq:nov14a6}
p^\prime(X)[H]=q^\prime(0)\quad\textrm{and}
\quad p^{\prime\prime}(X)[H]=q^{\prime\prime}(0)\,.
\end{equation}
The second formula in (\ref{eq:nov14a6}) is the evaluation of
the {\bf Hessian}, \index{Hessian}
$\text{Hess}\,(p)=p^{\prime\prime}(x)[h]$ of a polynomial $p \in \RRx$; it can be
thought of as the formal
second directional derivative of $p$ in the ``direction'' $h$.

If $p^{\prime\prime} \neq 0$, then 
the degree of $p$ is two or more, 
and the degree of $p^{\prime\prime}(x)[h]$ as a polynomial in the $2g$
variables $(x_1,\ldots,x_g, h_1\ldots,h_g)$ is equal to the degree
of $p(x)$ as a polynomial in $(x_1,\ldots,x_g)$
and is homogeneous of degree two in $h$. \index{degree of a polynomial}

The same conclusion holds for $k^{th}$ derivatives if $k\le d$,
the degree of $p$.
The expositions in \cite{HMVjfa} and in \cite{HPpreprint}
give more detail on the derivatives and the Hessian of
nc  polynomials.

\begin{example}\rm
A few concrete examples are listed for practice with the definitions, if
the reader is so inclined.
\begin{enumerate}
\item[\rm(1)] If $p(x)=x^4$, then \\
$p^\prime(x)[h] \ \ = hxxx + xhxx  + xxhx + xxxh$,\\
$p^{\prime\prime}(x)[h] \ \ =
2hhxx   + 2hxhx + 2 hxxh
  + 2xhhx   + 2 xhxh  + 2 xxhh$,
$
\\
p^{(3)}(x)[h  \ =
 6 (hhhx + hhxh + hxhh +xhhh)$, \\
$
p^{(4)}(x)[h]=
 24 hhhh
$
and $p^{(5)}(x)[h]=0$.
\vspace{2mm}
\item[\rm(2)] If $p(x)= x_2 x_1 x_2$, then $p^\prime(x)[h]=
h_2 x_1 x_2 +x_2 h_1 x_2 + x_2 x_1 h_2. $
\vspace{2mm}
\item[\rm(3)] If $p(x)=x_1^2 x_2$, then $p^{\prime\prime}(x)[h]
= 2(h_1^2 x_2 + h_1 x_1 h_2 + x_1 h_1 h_2)$.
\end{enumerate}
\end{example}


%




\subsubsection{The Signature of a non-commutative  quadratic }
 \label{subsec:DefSM}
Consider a symmetric polynomial
$q(x)[h]$
in the $2g$
variables $(x_1,\ldots,x_g, h_1\ldots,h_g)$ 
which is  homogeneous of degree two in $h$.
It
 admits a representation of the form
(a sum and difference of squares) 
\index{sum and difference of squares, SDS}
$$q(x)[h] =
\sum\limits^{\sigma_+}_{j=1} f^+_j(x)[h]^T f^+_j(x)[h] \ - \
\sum\limits^{\sigma_-}_{\ell=1}f^-_\ell (x)[h]^T f^-_\ell(x)[h]
\leqno{(SDS)}
$$ 
where $f^+_j(x)[h], f^-_\ell(x)[h]$  are non-commutative polynomials
which are homogeneous of degree one in $h$;
see e.g., Lemmas 4.7 and 4.8 of \cite{DHMjda} for details.
   Such representations are highly non-unique.
However, there is a unique smallest number 
of positive (resp., negative squares) 
$\smin_\pm(q)$ required in an SDS decomposition of $q$. 
These numbers are called
{\bf the signature of $q$}\index{signature of a polynomial}.
Later  $\smin_\pm(p^{\prime\prime})$
 will be identified with $\mu_{\pm}(\cZ)$, the number of
positive (resp., negative) eigenvalues of an appropriately chosen symmetric
matrix $\cZ$ appearing
in a Gram type representation of the Hessian $p^{\prime\prime}(x)[h]$ that is
discussed in Section \ref{sec:middle-border}.



\subsection{Previous results}
 \label{subsec:previous}
The number (or rather dearth)
of negative squares in an SDS decomposition of
the Hessian places serious
restrictions on the degree of an nc polynomial.
A theorem of Helton and McCullough [HM04]
states that a symmetric nc polynomial that is matrix convex has degree
at most two.  Since a symmetric nc
polynomial $p$ is matrix convex if $\sigma_-^{min}(p^{\prime\prime})=0$,
this is a special case of the
following more general result in \cite{DHMjda}.

\begin{theorem}
\label{thm:abs} If $p(x)$ is a symmetric nc polynomial of degree $d$
 in symmetric variables,  
then \beq
\label{eq:bigEst} d \le 2 \smin_\pm(p^{\prime\prime}) + 2.
\end{equation}
\end{theorem}

\subsection{Some Basic Definitions}
We next define a number of basic geometric objects associated to the nc
variety determined by an nc polynomial $p$.

\subsubsection{Varieties, tangent planes, and the second fundamental form}
 \label{subsec:ncsag}

The {\bf variety} (zero set) for $p$ is \index{variety, nc}
$$
\mathcal V(p):=\cup_{n\ge1}\mathcal{V}_n(p),
$$
where
$$
\mathcal{V}_n(p):=\left\{(X,v)\in\gtupn\times\RR^n:\,p(X)v=0\right\}.
$$

The {\bf clamped tangent plane} to $\mathcal V(p)$ at
\index{tangent plane, clamped}
$(X,v)\in\mathcal{V}_n(p)$ is
\begin{equation*}
  \tp{p}{X}{v}:= \{ H \in\Smatng: \  p^\prime(X)[H]v=0 \}.
\end{equation*}

\index{second fundamental form, clamped}
The {\bf  clamped second fundamental form}  for $\mathcal V(p)$
  at  $(X,v)\in\mathcal{V}_n(p)$ is the quadratic function
\begin{equation*}
  \tp{p}{X}{v} \ni H \mapsto
   - \langle p^{\prime\prime}(X)[H]v,v \rangle.
\end{equation*}

Note that
$$
\{ X \in \gtupn : \  (X,v) \in \mathcal V(p) \ \textrm{for some} \
v\ne 0 \}
  =\{X\in\gtupn: \  \det(p(X))=0\}
$$
is a variety in $\gtupn$ and typically has a  {\it true}
(commutative) tangent plane
at many points $X$,
which  of course  has codimension one, whereas
the clamped tangent plane at a typical point
$(X, v)\in \mathcal V_n(p)$ has codimension on the order of $n$ and
is contained inside the true tangent plane.

\subsubsection{Full rank points}
 \label{subsec:smooth}
  The point $(X,v)\in \mathcal V(p)$ is {\bf a full rank point}
  \index{full rank point} for $p$ if the mapping
$$
  \Smatng \ni H \mapsto p^\prime(X)[H]v \in \mathbb R^n
$$
  is onto. 
  %
  The full rank condition is   a non-singularity condition
  which amounts
  to a smoothness hypothesis. Such 
  conditions play a major role in real algebraic geometry,
  see Section 3.3  \cite{BCR91}.

As an example, consider the classical real
algebraic geometry case of $n=1$
 (and thus $X\in\mathbb R^g$)
with the commutative polynomial $\tilde{p}$ (which can be taken 
to be the {\bf commutative collapse} of the polynomial $p$).
In this case, a full rank point $(X, 1)\in\mathbb R^g\times\mathbb R$
is a point at which the gradient of $\tilde{p}$
does not vanish.
Thus, $X$ is a non-singular point for the zero variety of $\tilde{p}$.

Some  perspective for $n>1$ is obtained by counting dimensions.
If $(X, v)\in \Smatng\times\RR^n$, then $H\mapsto p^\prime(X)[H]v$
is a linear map
from the $g(n^2+n)/2$ dimensional space $\Smatng$ into
the $n$ dimensional space $\RR^n$.
Therefore, the codimension of
the kernel of this map is no bigger than $n$.
This codimension is $n$ if and only  if  $(X,v)$ is a full rank point
and in this case the clamped tangent plane has codimension $n$.

\subsubsection{Positive curvature}
 \label{subsec:curvature}
  As noted earlier, a notion of positive (really nonnegative)
  curvature can be defined in terms of
  the clamped second fundamental form.

The variety $\mathcal V(p)$ has {\bf positive curvature}
  \index{positive curvature} at $(X,v)\in \mathcal V(p)$
  if the clamped second fundamental form is nonnegative at $(X,v)$; i.e.,
  if
$$
\csff{p}{X}{H}{v} \ge 0\quad\textrm{ for every}\quad H\in\tp{p}{X}{v}\,.
$$

\subsubsection{Algebraically open sets}
 \label{subsec:alg-open}
  Define an {\bf algebraically open set} \index{algebraically open set}
  $\cO$ to be one of the form
  $\posdom{\cQ}:= \cap_{q\in \cQ} \{ X : \ q(X) \succ 0 \}$ where
  $\cQ$ is some  finite set of symmetric nc polynomials.
 Abusing  notation a little, we set
$$
  \mathcal V(p) \cap \cO:= \{ (X, v) : \ p(X)v=0 \ \textrm{and}
\ X \in \cO \}.
$$
  We are primarily motivated by the case where $\mathcal V(p)$ has
  positive
  curvature on an algebraically open set intersected with
  the full rank points of $p$ although the techniques produce
  a more general result of independent interest.

 \subsubsection{Irreducibility: The minimum degree defining polynomial
condition}
 \label{subsec:irreducible}
  While there is no tradition of what is an effective notion
  of  irreducibility for nc polynomials,
  there is a notion of minimal degree nc polynomial
  which is appropriate for the present context.

  In the commutative case
  the polynomial $p$ on $\RR^g$ is
  a {minimal degree defining polynomial} for $\mathcal V(p)$
  if there does not exist
  a polynomial $q$ of lower degree  such that
  $\mathcal V(p)= \mathcal V(q)$.
  This is a key feature of irreducible polynomials.

   There are several ways of generalizing the notion of
   minimal degree defining polynomial to the nc setting.
   In the next couple of paragraphs we
   describe the natural and effective
   generalization used here.

  Our main result is stated
  for certain subsets $\cS$ of $\cV(p)$, but the result remains
  of interest for $\cS$ equal all of $\cV(p)$. Thus, as we now turn
  to axiomatizing the admissible subsets $\cS$ of $\cV(p)$,
  the reader who is so
  inclined can simply choose $\cS=\cV(p)$  or
  $\cS$ equal to the full rank points of $\cV(p)$.

  A symmetric nc polynomial $p$ is a
  {\bf k-minimum degree defining polynomial for
$\cS \subseteq \mathcal V(p)$}
  \index{k-minimum degree defining polynomial} if
 \begin{enumerate}
  \item[\rm(1)] $\cS$ is a nonempty subset of $\mathcal V(p)$;

\item[\rm(2)]
\label{it:mindegKey}
 if $q\ne 0$ is another {\it (not necessarily symmetric)}
 nc   polynomial such that $q(X)v=0$ for each $(X,v)\in \cS$,
    then
   $$
\textup{deg}\,q(x)\ge\textup{deg}\,p(x) -k;
    $$

\item[\rm(3)]
\label{it:mindegBound}
there exists a {\it (not necessarily symmetric)}
nc polynomial  $q(x)$ with
$$
\textup{deg}\,q(x)=\textup{deg}\,p(x) -k
$$
such that $q(X)v=0$ for each pair
$(X,v)\in \cS$.

\end{enumerate}
A $0$-minimum degree defining nc polynomial
will be called a
{\bf minimum degree defining polynomial}.
\index{minimum degree defining polynomial}
Note this contrasts with  \cite{DHMind},
where minimal degree meant 1-minimal degree in terms of the present usage.

\subsection{Results I:  Positive curvature and the degree of $p$}
 \label{subsec:results1}

To ease the exposition,
 we start with a corollary of our main theorem. The proof of the
 corollary and the theorem will be supplied in Section 
 \ref{sec:pfofsigmain}.

    \begin{cor}
    \label{cor:sigmain}
     Let $p$ be a symmetric nc polynomial in symmetric variables,
      let $\cO$ be an algebraically open set
      and let $\cS$ denote the full rank points of $p$
      in $\mathcal V(p)\cap \cO.$
     If
     \begin{itemize}
        \item[(i)]  There exists $(X,v)\in(\mathbb R^g\times \mathbb R)\cap \cS;$ 
        \item[(ii)]  $\mathcal V(p)$
         has positive curvature at each point of $\cS$;
          and
        \item[(iii)] $p$ is a 0-minimum degree defining polynomial
          for $\cS$;
       \end{itemize}
         then the  degree of $p$ is at most two and $p$ is concave.
 \end{cor}

\subsection{The Signature of the Variety}
 \label{subsec:sig-variety}
   The fact that if $(X,v),(Y,w)\in \cV(p)$, then so is
   the direct sum
 \begin{equation*} 
\left( \textrm{diag}\,\{X,Y\}, \begin{bmatrix}
    v\\w\end{bmatrix}\right) =
 \left ( (\begin{bmatrix} X_1&0\\ 0 &Y_1\end{bmatrix}, \dots, \begin{bmatrix} X_g&0\\ 0 &Y_g\end{bmatrix}),\begin{bmatrix} v\\w\end{bmatrix}\right)
    = \left (\begin{bmatrix} X&0\\ 0 &Y\end{bmatrix}, \begin{bmatrix} v\\w\end{bmatrix}\right)
 \end{equation*}
  is a key feature of non-commutative varieties.
  This section begins with a discussion of direct sums.
  It continues with the notion of the signature of
  a quadratic form and concludes with our definition
  of the signature of $\cV(p)$.

\subsubsection{Direct Sums}
    Given a finite set $F=\{(X^1,v^1),\dots,(X^t,v^t)\}$ with
    $X^j \in \SMATNG{n_j}$ and
    $v^j \in \mathbb R^{n_j}$ for $j=1,2,\dots, t$,
    set
\begin{equation}
  \label{eq:XF}
     X_F= \textrm{diag}\,\{X^1,\ldots,X^t\}\quad\textrm{and}\quad
     v_F= \textrm{col}(v^1,\ldots,v^t)\,.
 \end{equation}
  Thus, if $q$ is an nc polynomial, then
 \begin{equation}
  \label{eq:wXv}
    q(X_F) v_F= \textup{col}\,
    (q(X^1) v^1,\ldots,q(X^t) v^t)\,.
 \end{equation}

   Now let
 $$
   \cS=\cup_{n=1}^\infty\cS_n\,,
 $$
    where
    $\cS_n \subseteq \Smatng \times \mathbb R^n$ for $n=1,2,\dots,$
    be given. 
    The 
    set $\cS$
    {\bf respects direct sums} \index{respects direct sums}
    if for each finite set
$$
  F=\{(X^1,v^1),\dots,(X^t,v^t)\}\quad\textrm{with}\quad
  (X^j,v^j)\in\cS_{n_j}\quad\textrm{and}\quad n=\sum_{j=1}^t n_j\,,
$$
 {\bf with repetitions allowed},
 the pair $(X_F, v_F)$ is in $\cS_n$.

  Examples of sets which respect direct sums include
 \begin{itemize}
  \item[(i)]  the zero
     set $\mathcal V(p)$ of an nc polynomial $p$
     (see equation \eqref{eq:wXv});
  \item[(ii)]  
     the set of full rank
     points of $p$ in $\mathcal V(p)$ (See Lemma \ref{lem:DirSmooth}); and
  \item[(iii)] the intersection of sets which respect direct sums.
 \end{itemize}
  Indeed, as remarked before, the reader may choose to take
  $\cS$ equal to either $\cV(p)$ or to the
  full rank points in $\cV(p)$.

\subsubsection{The Signature of a Quadratic Form}
  The Hessian $p^{\prime\prime}(x)[h]$  of an nc polynomial is
  a quadratic form in $h$. More generally,
  let $f(x)[h]$ be an nc symmetric polynomial in the $2g$ symmetric variables
  $x=(x_1,\ldots,x_g)$ and $h=(h_1,\ldots,h_g)$
  that is of degree $s$ in $x$
  and homogeneous of
  degree two  in $h$.
  Given a subspace $\mathcal H$ of
  $\Smatng$, let
\begin{equation*}
\begin{split}
e_\pm^n(X,v;f,\mathcal H)\quad\textrm{denote the
  maximum dimension of a strictly}\\
\textrm{positive/negative subspace
  of $\mathcal H$
  with respect to the quadratic form}
\end{split}
\end{equation*}
 \begin{equation}
  \label{eq:quadratic-form}
     \mathcal H \ni H \mapsto \langle f(X)[H]v,v \rangle.
 \end{equation}
   Here  strictly positive (resp., negative) subspace $\mathcal H$ means
$\langle f(X)[H]v,v\rangle >0$ (resp., $<0$)
   for $H\in \mathcal H$, $H\ne 0$. \index{$e_\pm$} \index{$c_\pm$}

\subsubsection{The Signature of the curvature of $\cV(p)$ relative to
$\cS$}
Given a symmetric nc polynomial $p$ in 
  symmetric variables
  and $(X,v)$, let
$$
\mathcal{T}=\{H\in\Smatng: \  p^\prime(X)[H]v=0\}
\quad\textrm{and}\quad
   c_\pm^n(X,v;p):=e_{\pm}^n(X,v;p^{\prime\prime},\mathcal T).
$$
When $p(X)v=0,$ so that  $(X,v)$ is in the zero set of $p$, then the
  subspace $\mathcal T$ is the {\it clamped tangent space.}

   The numbers $C_\pm(\cS)$, which are defined below in terms of 
  $c_\pm^n(X,v;p)$ in (\ref{eq:dec30c7}), bound the signature of the second
fundamental form of $\cV(p)$ on $\cS$. Because of the close
connection between the second fundamental form and the curvature of
$\cV(p)$ at a smooth point, we shall call  the numbers
${\cC}_\pm(\cS)$  the
{\bf signature of the curvature of
   $\cV(p)$ on $\cS$}. \index{signature of a variety}

   Note that if $\cS=\cup_{n\ge 1}\cS_n$ 
   is closed with respect to direct sums, then
   $$
\cS_n\ne\emptyset\ \textrm{for every integer} \ n\ge 1
\Longleftrightarrow \cS_1\ne\emptyset.
$$

\subsection{Results II}
 \label{subsec:results2}
    We can now state the main result of this article.

 \begin{theorem}
    \label{thm:sigmain}
     Let $p$ be a symmetric nc polynomial in $g$ symmetric variables,
      and let $\cS=\cup_{n\ge 1}\cS_n$ be 
            a subset of $\mathcal V(p)$  which respects direct sums
  and for which $\cS_1$ is nonempty. 
  \begin{itemize}
   \item[(A)]
       \begin{itemize}
        \item[(i)]  The limit
            \beq
\label{eq:dec30c7}
  C_\pm(\cS):=\lim_{n\uparrow \infty}\,\left(\sup
  \left\{\frac{c_\pm^n (X,v;p)}{n}: \ (X,v)\in\cS_n \right\}\right)
\end{equation}
exists.
        \item[(ii)] If  $p$ is a 0-minimum degree defining polynomial for $\cS$,
       \end{itemize}
         then
\begin{equation}
\label{eq:dec31a7}
 \textup{deg}\,p\le 2\, C_\pm(\cS) + 2.
\end{equation}

\smallskip

 \item [(B)] Moreover:
\begin{eqnarray*}
if \ C_-(\cS)&=&0,\ \text{then $p$ is a convex polynomial of degree $2$;}\\
if \ C_+(\cS)&=&0,\ \text{then $p$ is a concave polynomial of degree $2$.}
\end{eqnarray*}

\smallskip

\item[(C)] Conversely, if $p$ is convex (resp., concave), then $p$ has
     degree at most two and
     $c_-^n(X,v;p)=0$ (resp., $c_+^n(X,v;p)=0$)
      for every $(X,v)\in\Smatng\times \mathbb R^n$.

\smallskip

\item[(D)]
  Finally, if $\cS^\prime=\cup \cS^\prime_n$ is any other 
  subset of $\cV(p)$ for which 
  $\cS^\prime_1\ne \emptyset$ and $p$ is a $0$-minimal degree defining
  polynomial,  then
 $C_\pm(\cS)=C_\pm(\cS^\prime)$.
\end{itemize}
\end{theorem}


   In view of (D), the signature of $\cV(p)$ is determined on any
   nonempty subset
  $\cS$ of $\cV(p)$ which respects direct sums and is large enough so that $p$ is a minimal defining
  polynomial for $\cS$.  The next proposition shows how
  this  phenomenon carries to surprising extremes, 
  in that a  single pair $(X,v)$
  of high enough dimension often determines $\cC_\pm(\cV(p))$.  
These principles are elaborated upon in Theorem
 \ref{prop:localq-plus}.  Here we offer the following consequence of
 Theorem \ref{prop:localq-plus}.

\begin{proposition}
\label{prop:localq-simple}
  Let $p$ be a symmetric nc polynomial 
in  symmetric variables
   which  is a minimal degree defining polynomial for $\cV(p)$.
   Then   there is an integer $\hat n_{g,d}$
   not depending on $p$  such that  for every $n\ge \hat n_{g,d}$ 
  there exists a pair $(X,v) \in\gtupn \times \mathbb R^n$ in $\cV(p)$ 
 such that 
\[
  \cC_\pm( \cV(p) )=\lceil \frac{c_\pm^n(X,v;p)}{n}\rceil.
\]
  Here $\lceil r \rceil$ is the ceiling function; i.e.,
  the  smallest integer at least
  as large as the real number $r$. 
\end{proposition}



\subsection{Reader's Guide}
 \label{subsec:guide}
  The remainder of the paper is organized as follows.
  Section \ref{sec:examples} contains examples
  which illustrate Theorem \ref{thm:sigmain}. In
  Section \ref{sec:curvature} the relaxed Hessian is
  introduced and  the signature of
  the fundamental form (on the clamped tangent space)
  is shown to correspond closely to 
  the signature
  of this relaxed Hessian.  
  %
  %
  A key tool is a Gram like representation for nc
  quadratics, called the middle matrix-border vector
  representation, which is reviewed in Section 
  \ref{sec:middle-border}. 
   Section \ref{sec:direct} shows that the minimal degree hypothesis
  translates into a linear independence condition on
  the border vector.
%
This linear independence, via a ``CHSY Lemma'' 
(see   Section \ref{sec:CHSY}) is enough to put the signature
  of the relaxed Hessian 
  in close correspondence with that of the middle matrix.
  The signature of the middle matrix has been 
  carefully analyzed and exploited
  in studying issues of matrix convexity \cite{DHMjda}, \cite{DHMind},
  \cite{DGHM}, \cite{HM04} and this produces our
  main inequality \eqref{eq:dec31a7}. 
   The results outlined above are tied together in Section \ref{sec:pfofsigmain}
   to produce  the proof of Theorem \ref{thm:sigmain}.
  %
  The article concludes with  Section \ref{sec:convex-sets}
  which discusses the non-commutative analog of 
  the fact that the boundary of a convex 
  sublevel set in $\RR^g$ has nonnegative curvature.

  The paper is in principle self contained except for
 two previous results. One is the Middle Matrix Congruence Theorem
 from  \cite{DHMjda}. This congruence is given in equation
 \eqref{eq:may13a7}
 and related necessary results are summarized early in Section 
 \ref{sec:middle-border}.
 The other result is the CHSY Lemma \cite{CHSY03}, which
is stated and elaborated upon in Section \ref{sec:CHSY}.

\section{Examples}
 \label{sec:examples}
 In this section we compute some examples to illustrate
 the notation and objects from Theorem \ref{thm:sigmain}.

\subsection{A very simple example}
 In the following example, the
 null space
\begin{equation*}
  \mathcal T=\cT_p(X,v)=\{H\in\Smatng : \ p^\prime(X)[H]v=0\}
\end{equation*}
 is computed for certain choices of $p$, $X$, and $v$.
 Recall that if $p(X)v=0$, then the
  subspace $\mathcal T$ is the {\it clamped tangent plane}
  \index{clamped tangent plane} introduced in
  Subsection \ref{subsec:ncsag}.

\begin{example}\rm
 \label{ex:simple}
Let $X\in\RR^{\ntn}_{sym}$, $v\in\RR^n$, $v\ne 0$, let $p(x)=x^k$ for some
integer $k\ge 1$.
Suppose that $(X,v) \in \cV(p)$,
that is,  $X^kv=0$. Then, since
$$
X^kv=0\Longleftrightarrow Xv=0 \quad\textrm{when}\ X\in\RR^{\ntn}_{sym},
$$
it follows that $p$ is a minimum degree defining polynomial
for $\cV(p)$ if and only if $k=1$.

It is readily checked that
$$
(X,v) \in \cV(p)\Longrightarrow p^\prime(X)[H]v=X^{k-1}Hv,
$$
and hence that $X$ is a full rank point for $p$ if and only
 if $X$ is invertible.

Now suppose $k\ge 2$. Then,
$$
\langle p^{\prime\prime}(X)[H]v , v\rangle=2\langle HX^{k-2}Hv, v\rangle.
$$
Therefore, if $k>2$
$$
(X,v) \in \cV(p)  \quad \textrm{and} \quad p^\prime(X)[H]v=0
\ \ \Longrightarrow \  \ X Hv=0, \ \textrm{and so}
$$
$$
\langle p^{\prime\prime}(X)[H]v , v\rangle=0.
$$

To count the dimension of $\cT$  we can suppose without loss of generality
that
$$
X=\begin{bmatrix}0&0\\ 0&Y\end{bmatrix} \quad\text{and}\quad
v=\begin{bmatrix}1\\ 0\\ \vdots\\ 0\end{bmatrix},
$$
where $Y\in\RR^{(n-1)\times (n-1)}_{sym}$ is invertible.
Then, for the simple case under consideration,
$$
\cT=\{ H\in \RR^{\ntn}_{sym}: \  h_{21},\ldots, h_{n1}=0\},
$$
where $h_{ij}$ denotes the $ij$ entry of $H$. Thus,
$$
\textup{dim}\,\cT=\frac{n^2+n}{2}-(n-1),
\qquad
i.e.,
\qquad
\textup{codim}\,\cT=n-1.
$$
\end{example}

\begin{remark}\rm
We remark that
$$
X^k v=0\quad\textrm{and}\quad\langle p^{\prime\prime}(X)[H]v , v\rangle=0
\Longrightarrow p^\prime(X)[H]v=0\quad\textrm{if}\ k=2t\ge 4,
$$
as follows easily from the formula
$$
\langle p^{\prime\prime}(X)[H]v , v\rangle= 2\langle X^{t-1}Hv,
X^{t-1}Hv\rangle.
$$
\end{remark}

\subsection{Computation of $c_\pm$ and direct sums}
 We now turn to computing examples of the quantities
 $c_\pm^{n}(X,p,v)$ with special attention paid to their behavior
 under direct sums demonstrating the inequality
 of  Lemma \ref{lem:Hinequality}.

\begin{example}\rm
 \label{ex:less-simple}
 Specializing the previous example,
choose  $p(x)=x^3$ and suppose $X\in\RR^{n\times n}_{sym}$ satisfies
$X^2=I_n$.  Soon we will make a concrete choice of $X$
and consider various choices for the vector $v$.
To compute $c_\pm(X,v,p)$, we must compute
both the Hessian and the subspace $\mathcal T$.

 In this case,
$$
p^\prime(X)[H]=HX^2+XHX+X^2H\quad\textrm{and}\quad
p^{\prime\prime}(X)[H]=2(H^2X+HXH+XH^2).
$$
Next, upon imposing the supplementary constraint $X^2=I_n$, it is readily
seen that
$$
p^\prime(X)[H]v=0\ \ \Longleftrightarrow \ \ HXv=-2XHv
$$
and hence
\begin{equation}
\label{eq:dec23a8}
p^\prime(X)[H]v=0\Longrightarrow \langle p^{\prime\prime}(X)[H]v,v\rangle=
-6\langle XHv,Hv\rangle.
\end{equation}

To illustrate more detail, let
$$
X=\begin{bmatrix}1&0&0\\ 0&-1&0\\0&0&-1\end{bmatrix},\quad  v=
\begin{bmatrix}1\\ \beta\\ \gamma \end{bmatrix}\quad\text{and}\quad
\cT=\{H\in\RR^{3\times 3}_{sym}: \ p^\prime(X)[H]v=0\},
$$
Then it is easily checked that
$$
\cT=\textup{span}\left\{\begin{bmatrix}\beta^2&-3\beta&0\\-3\beta&1&0\\0&0&0
\end{bmatrix},\,
\begin{bmatrix}2\beta\gamma&-3\gamma&-3\beta\\-3\gamma&0&1\\-3\beta&1&0
\end{bmatrix},
\begin{bmatrix}\gamma^2&0&-3\gamma\\0&0&0
\\-3\gamma&0&1\end{bmatrix}\right\},
$$
and correspondingly,
\begin{eqnarray*}
\textup{span}\left\{Hv: H\in\cT\right\}&=&\textup{span}\left\{
\begin{bmatrix}\beta^2\\ \beta\\0\end{bmatrix},\,
\begin{bmatrix}2\beta\gamma\\ \gamma\\ \beta\end{bmatrix},\,
\begin{bmatrix}\gamma^2\\0\\ \gamma\end{bmatrix}\right\}\\
&=&\textup{span}\left\{
\beta\begin{bmatrix}\beta\\ 1\\0\end{bmatrix},\,
\beta\begin{bmatrix}\gamma\\ 0\\ 1\end{bmatrix}
+\gamma\begin{bmatrix}\beta\\ 1\\ 0\end{bmatrix},\,
\gamma\begin{bmatrix}\gamma\\0\\ 1\end{bmatrix}\right\}.
\end{eqnarray*}
Thus, if $0<\beta^2+\gamma^2$, then
$$
\textup{span}\left\{Hv: \  H\in\cT\right\}=\textup{span}\left\{
\begin{bmatrix}\beta\\ 1\\0\end{bmatrix},\,
\begin{bmatrix}\gamma\\0\\ 1\end{bmatrix}\right\},
$$
and, if 
\begin{equation*}
 Hv=\lambda \begin{bmatrix}\beta\\ 1\\0\end{bmatrix} +
   \mu \begin{bmatrix}\gamma\\0\\ 1\end{bmatrix},
\end{equation*}
it follows that
\begin{eqnarray*}
 \langle p^{\prime\prime}(X)[H]v,v\rangle&=&6\{\lambda^2+\mu^2-
(\lambda\beta+\mu\gamma)^2\} \\
&\ge& 6(\lambda^2+\mu^2)\{1-(\beta^2+\gamma^2)\} \ge 0,
\end{eqnarray*}
since $(\lambda\beta+\mu\gamma)^2\le (\lambda^2+\mu^2)(\beta^2+\gamma^2)$,
by the Cauchy-Schwartz inequality. Consequently,
$$
0<\beta^2+\gamma^2<1\Longrightarrow
c_-^n(X,v;p)=0\quad\text{and}\quad c_+^n(X,v;p)=2=\mu_-(X).
$$

 We now turn to the behavior of $c_\pm$ under direct sums.
 Supposing again only that $X^2=I$ (and allowing for general $n$) let
$$
Y=\textup{diag}\{X,\ldots,X\}\quad\textrm{and}\quad
w=\textup{col}(v,\ldots,v) \quad k\ \textup{times},
$$
then $Y^2=I_{kn}$ and hence
$$
p^\prime(Y)[H]w=0\Longleftrightarrow HYw=-2YHw,
$$
$$
p^\prime(Y)[H]w=0\Longrightarrow \langle p^{\prime\prime}(Y)[H]w,w\rangle=
-6\langle YHw,Hw\rangle,
$$
and
$$
\cT=\{H\in\RR^{kn\times kn}_{sym}: \  HYw+2YHw=0\},
$$
 just as before.

  Let
\begin{equation}
\label{eq:dec22a8}
\cD_k=\{H\in\RR^{kn\times kn}_{sym}: \  H=\textup{diag}\{H^1,\ldots,H^k\}\
\text{with}\ H^j\in\RR^{n\times n}_{sym}\}
\end{equation}
and
\begin{equation}
\label{eq:dec22b8}
{\cT}_{\cD_k}=\cT\cap\cD_k.
\end{equation}
Then clearly $\cT\supseteq {\cT}_{\cD_k}$ and
$\{Hv: \ H\in\cT\}\supseteq \{Hv: \ H\in {\cT}_{\cD_k}\}$. Therefore,
\begin{equation}
\label{eq:dec22c8}
c_\pm^{kn}(Y,w;p)\ge kc_\pm^{n}(X,v;p),
\end{equation}
 an inequality which holds generally, see Lemma \ref{lem:Hinequality}.
\end{example}

 A finer analysis of a specialization of the previous
 example shows that the inequality in equation (\ref{eq:dec22c8})
 (and Lemma \ref{lem:Hinequality}) can
 be strict.

\def\pp{r}
\begin{example}\rm
 \label{ex:elaborate}
Let $p(x)=x^3$ and
$X=\textup{diag}\{I_\pp,-I_q\}$ with $\pp\ge 1$, $q\ge 1$ and
$\pp+q=n$ (so that $X^2=I_n$ as in the previous example)
and correspondingly partition
$H\in\RR^{n\times n}_{sym}$ and $v\in\RR^n$ as
$$
H=\begin{bmatrix}H_{11}&H_{12}\\ H_{21}&H_{22}\end{bmatrix}, \quad
v=\begin{bmatrix}v_1\\v_2\end{bmatrix}
$$
with $H_{11}\in\RR^{\pp\times \pp}_{sym}$, $H_{22}\in\RR^{q\times q}_{sym}$,
$v_1\in\RR^\pp$ and $v_2\in\RR^q$. Then
$$
\cT=\{H\in \RR^{n\times n}_{sym}:\,3H_{11}v_1=-H_{12}v_2\ \text{and}\
H_{21}v_1=-3H_{22}v_2\}.
$$
Thus, if $H\in\cT$, then
$$
Hv=\begin{bmatrix}H_{11}v_1+H_{12}v_2\\ H_{21}v_1+H_{22}v_2\end{bmatrix}=
-2\begin{bmatrix}H_{11}v_1\\ H_{22}v_2\end{bmatrix}=
\frac{2}{3}\begin{bmatrix}H_{12}v_2\\ H_{21}v_1\end{bmatrix}
$$
and
$$
3v_1^T H_{11}v_1=-v_1^T H_{12}v_2=-v_2^T H_{21}v_1=3v_2^T H_{22}v_2.
$$
%
Therefore, since $H_{12}$ is an arbitrary $\pp\times q$ matrix and
$H_{21}=H_{12}^T$ ,
$$
\{Hv: \  H\in\cT\}=\left\{\begin{array}{l}
\left\{\begin{bmatrix}Av_2\\A^Tv_1\end{bmatrix}: \ A\in\RR^{\pp\times q}\right\}
\quad\textrm{if}\quad v_1\ne 0\ \text{and}\ v_2\ne 0\\ \\
0\quad\textrm{if}\quad v_1=0\ \text{or}\ v_2=0.\end{array}\right.
$$
If, say, $\pp=2$, $q=3$, $v_1^T=\begin{bmatrix} a & b\end{bmatrix}\ne 0$
and $v_2^T=\begin{bmatrix} c & d & e \end{bmatrix}\ne 0$, then
$$
\{Hv: \  H\in\cT\}=\textup{span}\left\{
\begin{bmatrix}c\\0\\a\\0\\0\end{bmatrix},\,
\begin{bmatrix}d\\0\\0\\a\\0\end{bmatrix},\,
\begin{bmatrix}e\\0\\0\\0\\a\end{bmatrix},\,
\begin{bmatrix}0\\c\\b\\0\\0\end{bmatrix},\,
\begin{bmatrix}0\\d\\0\\b\\0\end{bmatrix},\,
\begin{bmatrix}0\\e\\0\\0\\b\end{bmatrix}\right\}.
$$
If  $b=c=d=0$, then
$$
\{Hv: \  H\in\cT\}=\textup{span}\left\{
\begin{bmatrix}0\\0\\a\\0\\0\end{bmatrix},\,
\begin{bmatrix}0\\0\\0\\a\\0\end{bmatrix},\,
\begin{bmatrix}e\\0\\0\\0\\a\end{bmatrix},\,
\begin{bmatrix}0\\e\\0\\0\\0\end{bmatrix}\right\}.
$$
If also $a>e>0$, then this span
splits into the orthogonal sum of
a negative space of dimension one
and a positive space of dimension three with
respect to the bi-linear form
$\langle p^{\prime\prime}(X)[H]v,v\rangle$.
Correspondingly (in view of (\ref{eq:dec23a8})),
$$
c_-^5(X,v;p)=1 \quad \text{and} \quad  c_+^5(X,v;p)=\mu_-(X).
$$

Turn now to direct sums.
If
$$
Y=\textup{diag}\{X,\ldots,X\}\quad\textrm{and}\quad
w=\textup{col}(v,\ldots,v) \quad k\ \textup{times},
$$
then
$$
c_\pm^{kn}(Y,w;p)=c_\pm^{kn}(\widetilde{Y},\widetilde{w};p),
$$
where
$$
\widetilde{Y}=\textup{diag}\{ I_{kr}, -I_{kq}\},\quad \widetilde{w}
=\textup{col}
(w_1,w_2),\quad w_j=\textup{col}(v_j,\ldots,v_j)\quad\text{for}\ j=1,2.
$$
Thus,
$$
\{H\widetilde{w}: \ H\in\cT\}
=\left\{\begin{array}{l}
\left\{\begin{bmatrix}Aw_2\\A^Tw_1\end{bmatrix}:\
A\in\RR^{k \pp\times kq}\right\}
\quad\textrm{if}\quad w_1\ne 0\ \text{and}\ w_2\ne 0\\ \\
\ \ \ \ \ \ \ \ 0 \qquad    \quad  \qquad \qquad  \qquad  \textrm{if}\quad w_1=0\ \text{or}\ w_2=0.
\end{array}
\right.
$$
If, $\pp=2$, $q=3$, $k=2$, $v_1^T=\begin{bmatrix} a & b\end{bmatrix}\ne 0$
and $v_2^T=\begin{bmatrix} c & d & e \end{bmatrix}\ne 0$, then
$$
c_-^{10}(Y,w;p)=3\quad\text{and}\quad c_+^{10}(Y,w;p)=\mu_-(Y)=2\mu_-(X).
$$

 In particular, with $k=5$ and $n=2$,
\begin{equation*}
  c_-^{kn}(Y,w;p) > kc_-^{n}(X,v;p).
\end{equation*}
\end{example}


\section{Curvature: The Hessian on a Tangent Plane vs. the Relaxed Hessian}
 \label{sec:curvature}

 Our main tool for analyzing the curvature of non-commutative real varieties
 is a variant of the Hessian for symmetric nc polynomials $p$ of
 degree $d$ in $g$ non-commuting variables.
The
curvature of $\cV(p)$ is defined in terms of $\textup{Hess}\,(p)$
compressed to
tangent planes, for each dimension $n$.
This compression of the Hessian is awkward to work with directly, and
so we associate to it a quadratic polynomial $q(x)[h]$ (defined
for all $H\in \Smatng$, not just $H\in \tp{p}{X}{v}$)
called the relaxed Hessian.

  Let $V_k(x)[h]$ denote the vector of polynomials with entries
$h_jw(x)$, where $w(x)$ runs through the set of $g^k$ words of length
$k$,
$j=1,\ldots,g$. Although the order of the entries is fixed in some of our
earlier applications (see e.g., formula (2.3) in \cite{DHMjda}) it is
irrelevant for
the
moment.
Thus, $V_k=V_k(x)[h]$ is a vector of height $g^{k+1}$, and the vectors
\begin{equation}
\label{eq:may14a7}
V(x)[h]=\textup{col}(V_0,\ldots,V_{d-2})\quad \textrm{and}\quad
\widetilde{V}(x)[h]=\textup{col}(V_0,\ldots,V_{d-1})
\end{equation}
are vectors of height $g\alpha_{d-2}$ and $g\alpha_{d-1}$, respectively, where 
\begin{equation}
 \label{eq:sep18a7}
  \alpha_t=1+g+\cdots +g^t. 
\end{equation}
Note that
$$
  \widetilde{V}(x)[h]^T \widetilde{V}(x)[h] = \sum_{j=1}^g
  \sum_{|w|\le d-1} \; w(x)^T h_j^2 w(x),
$$
 where $|w|$ denotes the degree (length) of the word $w$.

 The {\bf relaxed Hessian} \index{relaxed Hessian}
 of the symmetric nc polynomial
 $p$ of degree $d$ is defined to be the polynomial
\begin{equation}
 \label{eq:relaxedHessian}
   p^{\prime\prime}_{\lambda, \delta}:=
      p^{\prime\prime}(x)[h]+ \delta \, \widetilde{V}(x)[h]^T
\widetilde{V}(x)[h]
         +\lambda \,  p^{\prime}(x)[h]^T p^\prime(x)[h].
\end{equation}

  Suppose $X\in \gtupn$ and $v\in \mathbb R^n$.
  We say that the {\bf  relaxed Hessian is positive}
  \index{relaxed Hessian, positive} at $(X,v)$ if
  for each $\delta>0$ there is a $\lambda_\delta>0$ so that
  for all $\lambda>\lambda_\delta$
$$
   0\le
    \langle p^{\prime\prime}_{\lambda, \delta}(X)[H]v,v\rangle
$$
 for all $H\in\gtupn$.
 Correspondingly we say that the {\bf  relaxed Hessian is negative}
 \index{relaxed Hessian, negative} at
$(X,v)$ if
 for each $\delta<0$ there is a $\lambda_\delta<0$ so that
 for all $\lambda \le \lambda_\delta$,
$$
   0\le
    -\langle p^{\prime\prime}_{\lambda, \delta}(X)[H]v,v\rangle
$$
 for all $H\in\gtupn$.
 Given a subset $\cS=\cup_{n=1}^\infty \cS_n$, with
$\cS_n\subseteq(\gtupn \times \mathbb R^n)$,
 we say that the
{\bf relaxed Hessian is positive (resp., negative) on $\cS$}
 \index{relaxed Hessian, positive on $S$} if
 it is positive (resp., negative) at each $(X,v)\in \cS$.

\begin{example} \rm
 \label{ex:classical}
Consider the classical $n=1$ case.
Suppose that $p$ is strictly smoothly quasi-concave, meaning that
all superlevel sets of $p$ are strictly convex with
strictly positively curved smooth boundary $\nu$.
Suppose that the gradient $\nabla p$ (written as a row vector) never vanishes on 
$\RR^g$. Then 
$G = \nabla p (\nabla p)^T$ is strictly positive, and at each point $X$ in
$\RR^g$ the relaxed Hessian  can be decomposed as a block matrix
subordinate to the tangent plane
to the level set at $X$
and to its orthogonal complement (the gradient direction).
In this decomposition the
relaxed Hessian with $\delta=0$ has the form
$$ R= \left[%
\begin{array}{cc}
  A & B \\
  B^T & D+ \lambda G \\
\end{array}%
\right],
$$
where, by convention
the second fundamental form is $A$
or $-A$, depending on the rather arbitrary choice of inward
or outward normal to $\nu$. If we select
our  normal direction to be $\nabla p$,
then $-A$ is the classical second fundamental form
as is consistent with the choice of sign in our definition in Subsection 
\ref{subsec:curvature}. (All this concern with the sign
is irrelevant to the content of this paper
and can be ignored by the reader.)

Next, in view of the presumed strict positive
curvature  of $\nu$, the matrix $A$
at each point of $\nu$ is  negative definite.
Thus, by standard Schur complement arguments,
$R$ will be negative
definite on any compact region of $\RR^g$ if
$$
D+\lambda G-B^TA^{-1}B\prec 0
$$
on this region.
Thus, strict convexity assumptions on the sublevel sets of
$p$ make the relaxed Hessian negative
definite even if $\delta=0$.

In the non-commutative case,
Remark \ref{rem:delta} (below) implies that 
if $n$ is large enough, then the second fundamental form
will have a nonzero null space.
Consequently, in this paper we shall
be forced to consider the case where $A$ is negative
semi-definite and
has a nonzero null vector $\eta$.
Then, to obtain negative definite $R$, we must add
another negative term, say $\delta I$, with arbitrarily small
$\delta < 0$.
After this, the argument based on
choosing $-\lambda$ large
succeeds as before. This $\delta$ term plus the $\lambda$
term produces the relaxed Hessian,
and proper selection of these terms make it negative definite.


Some additional detail on the connection between convexity of
a sublevel set and non-negativity of the relaxed Hessian
(positive curvature) is
provided in Section \ref{sec:convex-sets}.


\end{example}

\def\tcld{\widetilde{c}^{\lambda,\delta}}
\def\tcldn{\widetilde{c}^{\lambda,\delta, n}}

 The following theorem provides a link
 between the signature of the clamped
 second fundamental form with that of the
 Hessian.

\begin{theorem}
 \label{thm:signature-clamped-relaxed}
    Suppose $p$ is a symmetric nc polynomial of degree $d$ in $g$
    symmetric variables and $(X,v)\in
\Smatng \times\mathbb R^n$.
\begin{enumerate}
\item[\rm(1)] There exists $\delta_0>0$
    such that for each $\delta\in (0,\delta_0]$  there exists a
$\lambda_\delta >0$ so that for every $\lambda\ge \lambda_\delta$,
 \begin{equation*}
    e_-^n(X,v;p^{\prime\prime}_{\lambda,\delta}, \Smatng) = c_-^n(X,v;p)
 \end{equation*}
\item[\rm(2)]  There exists a $\delta_0<0$
    such that for  each $\delta\in [\delta_0, 0)$ there exists a
$\lambda_\delta<0$ so that for every $\lambda \le \lambda_\delta$,
 \begin{equation*}
    e_+^n(X,v;p^{\prime\prime}_{\lambda,\delta}, \Smatng) = c_+^n(X,v;p).
 \end{equation*}
\item[\rm(3)]
  If $\cV(p)$ has  positive  curvature at
   $(X,v)\in \mathcal V_n(p)$,
   i.e., if
$$
\langle p^{\prime\prime}(X)[H]v,v\rangle\le 0\quad\text{for every}\ H\in
\cT_p(X,v),
$$
  then $c_+^n(X,v;p)=0$ and for every $\delta<0$
   there exists a
$\lambda_\delta<0$ such that for all $\lambda \le \lambda_\delta$,
$$
\langle p_{\lambda,\delta}^{\prime\prime}(X)[H]v,v\rangle\le 0
\quad\text{for every}\ H\in \Smatng;
$$
i.e.,  the relaxed Hessian of $p$ is negative at $(X,v)$.
\end{enumerate}
\end{theorem}
Note: (1) and (2) do not require $(X,v)$ to be in $\cV_n(p)$.

The proof employs a variant of the Hessian which we now introduce.
  Let $p^{\prime\prime}(X)[H][K]$
  denote the matrix obtained by differentiating
  $p^\prime(X)[H]$ in the direction
  $K\in\Smatng$; i.e.,
$$
  p^{\prime\prime}(X)[H][K]= \lim_{t\to 0} \frac{1}{t}(p^\prime(X+tK)[H]
- p^\prime(X)[H]).
$$
  In particular,
$$
p^{\prime \prime}(X)[H]=p^{\prime\prime}(X)[H][H]\quad
\textrm{and}\quad
p^{\prime \prime}(X)[H][K]=p^{\prime\prime}(X)[K][H]\,.
$$

\smallskip

\begin{proof}
 Let $d$ denote the degree of $p$ and $g$ the number of non-commutative
 symmetric  variables. 
To verify (1), let $\mathcal H=\Smatng$
endowed with the  Hilbert Schmidt norm:
\begin{equation}
\label{eq:jan7a9}
\langle H, K\rangle_{\cH}=\textup{trace}\,K^TH=\sum_{j=1}^g\textup{trace}\,
K_jH_j\,.
\end{equation}
The mapping
$$
  \mathcal H\times \mathcal H \ni (H,K)
\mapsto \langle p^{\prime\prime}(X)[H][K]v,v \rangle_{\RR^n}
$$
   is bi-linear and, because $\mathcal H$ is finite dimensional, bounded.
   Thus, there is a bounded linear
   operator $A$ on $\mathcal H$ so that
\begin{equation}
\label{eq:dec30b7}
   \langle A H, K\rangle_{\mathcal H} =
      \langle p^{\prime\prime}(X)[H][K]v,v \rangle_{\mathbb R^n}.
\end{equation}
    The operator  $A$ is selfadjoint with respect to this inner product,
because
$$
  p^{\prime\prime}(X)[H][K] = p^{\prime\prime}(X)[K][H].
$$

    Similarly, there are bounded linear selfadjoint operators $Q$ and
$E$ on $\cH$ so that
\begin{equation*}
 \begin{split}
    \langle QH,K \rangle_{\cH} = &
\langle p^{\prime}(X)[H]v,p^{\prime}(X)[K]v \rangle_{\RR^n} \\
    \langle EH,K \rangle_{\cH} = &
\langle \widetilde{V}(X)[H]v, \widetilde{V}(X)[K]v\rangle_{\RR^n}\,.
 \end{split}
\end{equation*}
Thus,
$$
\langle p^{\prime\prime}_{\lambda, \delta}(X)[H]v,v\rangle_{\RR^n}=
\langle(A+\lambda Q+\delta E)H, H\rangle_{\cH}
$$
 for all $H\in\gtupn$.
Let
$$
  \mathcal N=\{H\in\mathcal H: \ \widetilde{V}(X)[H]v=0\}.
$$
 Note that $H\in\mathcal N$ if and only if $H_jw(X)v=0$
 for  $j=1,\ldots,g$ and all words $w$ of degree at most $d-1$.
 In particular,
 $A{\cN}=\{0\}$, $Q{\cN}=\{0\}$ and $E{\cN}=\{0\}$, and
 thus  it suffices
 to focus on $\cN^\perp$.

Let
\begin{eqnarray*}
 \mathcal M&=&\{H\in \mathcal N^\perp: \ \,
 p^\prime(X)[H]v=0\}
\end{eqnarray*}
and
\begin{eqnarray*}
  {\cL}&=&{\cM}^\perp\cap\cN^\perp\,.
\end{eqnarray*}
Then
$$
 {\cN}^\perp ={\cM}\oplus {\cL}\,.
$$
Now let  $\mathcal M_-$ denote the span of the eigenspaces corresponding
   to the negative eigenvalues of the compression
    of $A$ to $\mathcal M$ (i.e., of $P_{\cM}A\vert_{\cM}$), let
$\mathcal M_+$ denote the orthogonal complement of $\mathcal M_-$ in
$\mathcal M$ and observe that:
\begin{enumerate}
  \item [\rm(a)]  $\textup{dim}\  \mathcal M_-
  \  = c_-^n(X,v;p);$
      \item[\rm(b)] $Q$ is strictly positive definite on $\mathcal L$
       and is $0$ on $\mathcal M;$
  \item[\rm(c)] $E$ is strictly positive definite on
$\mathcal N^\perp=\mathcal M\oplus  \mathcal L;$
\item[\rm(d)] $P_{\cM_+}A\vert_{\cM_+}\succeq 0$ and
$P_{\cM_-}A\vert_{\cM_-}\prec 0$.
\end{enumerate}
%
   Hence, there is a $\delta_0 > 0$ so that
   if $0<\delta \le \delta_0$, then the compression
   of $A+\delta E$ to $\mathcal M_-$ is negative definite and
   the  compression of $A+\delta E$ to $\mathcal M_+$ is positive
   definite. Therefore, if  $\lambda >0$ is sufficiently large,
   the compression of $A+\delta E + \lambda Q$ to $\mathcal M_+\oplus \cL$
   is positive definite; whereas its compression to $\mathcal M_-$ is
   equal to the compression of $A+\delta E$ to $\mathcal M_-$, which is
   negative definite. Thus, as
$$\mathcal N^\perp =\mathcal M_- \oplus (\mathcal M_+\oplus \mathcal L),$$
   it now follows that $A+\delta E +\lambda Q$ has $c_-(X,v;p)$ negative
   eigenvalues (counting with multiplicity).

  The proof of (2) is similar to the proof of (1).
To prove (3), fix $\delta_0<0$ so that (2) holds and
choose $\delta$ so that
$\delta_0\le\delta <0$. Then
  there is a $\lambda_\delta$ satisfying the conclusion of (2);
  i.e., $e_+^n(X,v;p^{\prime\prime}_{\lambda,\delta}, \Smatng)=0$
  and thus (3) holds.  Moreover, if $\delta_*<\delta$, then
  it is still the case that
  $e_+^n(X,v;p^{\prime\prime}_{\lambda,\delta_*}, \Smatng)=0$
  and hence (3) holds for all $\delta<0$.
\end{proof}

\begin{remark}
\label{rem:dec30a07}
Let $(\cH_A)_-$ (resp., $(\cH_A)_0$, $(\cH_A)_+$) denote the span of the
eigenvectors of $A$ corresponding to negative (resp., zero, positive)
eigenvalues of the operator $A$ that was introduced in the proof of
Theorem \ref{thm:signature-clamped-relaxed}. Then
\begin{equation}
\label{eq:dec30a7}
\Smatng=(\cH_A)_-\oplus (\cH_A)_0 \oplus (\cH_A)_+
\end{equation}
In particular, $(\cH_A)_-$ is a maximal strictly negative subspace of
$\Smatng$ with respect to the indefinite inner product
(\ref{eq:dec30b7}), and $(\cH_A)_0 \oplus (\cH_A)_+$ is complementary to it.
\end{remark}

\subsection{Example illustrating
Theorem \ref{thm:signature-clamped-relaxed}}
 \label{subsec:curvature-example}
In this subsection we continue
with Example \ref{ex:elaborate}
to illustrate the ingredients in the
 proof of Theorem \ref{thm:signature-clamped-relaxed}.

\begin{example}
 \label{ex:proof-sig-relaxed}
 \rm
  Let $p(x)=x^3$, let $X=\textup{diag}\{I_2,-I_3\}$, and let
$v^T=\begin{bmatrix}a&0&0&0&e\end{bmatrix}$
with $a\ge e> 0.$

Let $u_j$ denote the $j^{th}$ standard basis vector for $\RR^5$
with
$j=1,\ldots,5$  and let $S_{ij}$ denote the
normalized symmetrized elementary matrices
$$
S_{ij}=\left\{\begin{array}{l}(u_iu_j^T+u_ju_i^T)/\sqrt{2}\quad
\textrm{for}\quad i\ne j\\ \\
u_iu_i^T \quad\textrm{for}\quad i=j.\end{array}\right.
$$
Then the set of 15 matrices
$$
\{S_{ij}: \  i,j=1,\ldots,5\quad \text{and}\quad i\le j\}
$$
is an orthonormal basis for $\RR^{5\times 5}_{sym}$
with respect to the trace
inner product (\ref{eq:jan7a9}).
In terms of the notation of Theorem \ref{thm:signature-clamped-relaxed},
$$
\cN=\{H\in\RR^{5\times 5}_{sym}: \ \widetilde{V}(X)[H]v=0\}
=\textup{span}\{S_{22}, S_{23}, S_{24}, S_{33}, S_{34}, S_{44}\},
$$
$$
\cN^\perp=\textup{span}\{S_{11},S_{21}, S_{31}, S_{41}, S_{51}, S_{25},
S_{35}, S_{45}, S_{55}\}
$$
and
$$
\cM=\{H\in\cN^\perp: \ p^\prime(X)[H]v=0\}
=\textup{span}\{\gh_1,\gh_2,\gh_3,
\gh_4\},
$$
where
\begin{eqnarray*}
\gh_1&=&-eS_{11}+3a\sqrt{2}S_{15}-\frac{a^2}{e}S_{55}\\
\gh_2&=& eS_{21}-3aS_{25}\\
\gh_3&=& 3eS_{31}-aS_{35}\\
\gh_4&=& 3eS_{41}-aS_{45}.
\end{eqnarray*}
Moreover,
$$
\cL:=\cM^\perp\cap\cN^\perp=\textup{span}
\{\gh_5, \gh_6,\gh_7, \gh_8, \gh_9\},
$$
where
\begin{eqnarray*}
\gh_5&=&3aS_{21}+eS_{25}\\
\gh_6&=&aS_{31}+3eS_{35}\\
\gh_7&=&aS_{41}+3eS_{45}\\
\gh_8&=&6aS_{11}+\sqrt{2}eS_{15}\\
\gh_9&=&-a^2S_{11}+\frac{3\sqrt{2}a^3}{e}S_{15}+(e^2+18a^2)S_{55}
\end{eqnarray*}
and
$$
\textrm{trace}\,\gh_j^T\gh_i=0\quad\text{for}\quad i,j=1,\ldots,9\quad
\text{if}\quad i\ne j.
$$
Thus, as
$$
\begin{array}{ll}
\gh_1v=2aeu_1+2a^2u_5&\gh_1Xv=-4aeu_1+4a^2u_5\\
\gh_2v=-\sqrt{2}aeu_2 & \gh_2Xv=2\sqrt{2}aeu_2\\
\gh_3v=\sqrt{2}aeu_3 & \gh_3Xv=2\sqrt{2}aeu_3\\
\gh_4v=\sqrt{2}aeu_4 & \gh_4Xv=2\sqrt{2}aeu_4,
\end{array}
$$
it is readily confirmed that if $H=\sum_{j=1}^4\alpha_j\gh_j$, then
$$
p^\prime(X)[H]v=(2H+XHX)v=0
$$
and
\begin{eqnarray*}
\langle p^{\prime\prime}(X)[H]v,v\rangle&=&-6\langle XHv,Hv\rangle\\
&=&12\{2\alpha_1^2a^2(a^2-e^2)-\alpha_2^2a^2e^2+\alpha_3^2a^2e^2+
\alpha_4^2a^2e^2\}.
\end{eqnarray*}
Consequently,
$$
\cM_-=\textup{span}\{\gh_2\}\quad\textrm{and}\quad
\cM_+=\cM\ominus\cM_-=\textup{span}\{\gh_1, \gh_3,
\gh_4\}.
$$
Note that if $a>e$, then $\cM_+$ contributes
three positive squares, whereas,
if $a=e$, then only two.

Next we look at the $\lambda$ and $\delta$ terms
used in the relaxed Hessian.
Since
$$
\begin{array}{ll}
\gh_5v=\frac{1}{\sqrt{2}}(3a^2+e^2)u_2&\gh_5Xv
=\frac{1}{\sqrt{2}}(3a^2-e^2)u_2\\
\gh_6v=\frac{1}{\sqrt{2}}(a^2+3e^2)u_3&\gh_6Xv
=\frac{1}{\sqrt{2}}(a^2-3e^2)u_3\\
\gh_7v=\frac{1}{\sqrt{2}}(a^2+3e^2)u_4&\gh_7Xv
=\frac{1}{\sqrt{2}}(a^2-3e^2)u_4\\
\gh_8v=(6a^2+e^2)u_1+aeu_5&\gh_8Xv
=(6a^2-e^2)u_1+aeu_5\\
\gh_9v=2a^3u_1+(\frac{3a^4}{e}+e^3+18a^2e)u_5&\gh_9Xv=
-4a^3u_1+(\frac{3a^4}{e}-(e^3+18a^2e))u_5,
\end{array}
$$
it is readily checked that if $H=\sum_{j=1}^9\alpha_j\gh_j$, then
\begin{equation*}
\begin{split}
p^\prime(X)[H]v&=\alpha_5\frac{1}{\sqrt{2}}(9a^2+e^2)u_2
+\alpha_6\frac{1}{\sqrt{2}}(a^2+9e^2)u_3
+\alpha_7\frac{1}{\sqrt{2}}(a^2+9e^2)u_4\\
&{}+\alpha_8[(18a^2+e^2)u_1+aeu_5]+\alpha_9(3\frac{a^4}{e}+3e^3+54a^2e)u_5
\end{split}
\end{equation*}
and hence that
\begin{equation*}
\begin{split}
\langle p^\prime(X)[H]v,p^\prime(X)[H]v\rangle
&=\frac{1}{2}\alpha_5^2(9a^2+e^2)^2
+\frac{1}{2}\alpha_6^2(a^2+9e^2)^2+\frac{1}{2}\alpha_7^2(a^2+9e^2)^2\\
&{}+\alpha_8^2(18a^2+e^2)^2+(\alpha_8ae
+\alpha_9(3\frac{a^4}{e}+3e^3+54a^2e))^2.
\end{split}
\end{equation*}

Now we calculate the $\delta$ term of the relaxed Hessian.
The preceding formulas for $\gh_j$, $j=1,\ldots,9$ and the fact
that $Xu_j=u_j$ for $j=1,2$ and $Xu_j=-u_j$ for $j=3,4,5$ imply that if
$H=\sum_{j=1}^9\alpha_j\gh_j$,  then
$$
Hv=\sum_{j=1}^9\alpha_j\gh_jv=\sum_{j=1}^5\beta_ju_j\quad\textrm{and}\quad
HXv=\sum_{j=1}^5\gamma_ju_j
$$
for appropriately chosen constants
$\beta_1,\ldots,\beta_5$ and
$\gamma_1,\ldots,\gamma_5$; i.e.,
$$\beta_2=-\sqrt{2}ae\alpha_2+\frac{1}{\sqrt{2}}(3a^2+e^2)\alpha_5,
\qquad \qquad
\beta_3=\sqrt{2}ae\alpha_3+\frac{1}{\sqrt{2}}(a^2+3e^2)\alpha_6, $$
 etc.
Thus,
\begin{eqnarray*}
\langle \widetilde{V}(X)[H]v, \widetilde{V}(X)[H]v\rangle&=&
2v^THHv+v^TXHHXv\\
&=&\sum_{j=1}^5(2\beta_j^2+\gamma_j^2).
\end{eqnarray*}
\end{example}

\section{Direct sums}
 \label{sec:direct}
   The next lemma expresses a basic principle
   \cite{CHSY03} \cite{HMVjfa} that provides a link between the direct
   sum and the minimum degree (irreducibility)
    hypotheses in Theorem \ref{thm:sigmain}.
   Also in the section is the observation that the full rank
   condition, as defined in Section \ref{subsec:smooth}, is preserved
   under direct sums.

\begin{lemma}
  \label{lem:directsumsystem}
Suppose that the 
    set $\cS=\cup_{n\ge 1}{\cS}_n$ respects  direct sums,
      let $N$ be a given positive integer and let ${\cW}_N$ denote
    the set of all words of length at most $N$.   Then, either,
 \begin{itemize}
    \item[(E)]
     \label{item:E}
        there exists a positive integer $n$
       and  a pair $(X, v)\in\cS_n$ such that the set
      \begin{equation*}
       \{ w(X)v: \  w  \in \mathcal W_N\}
      \end{equation*}
            is a linearly independent set
      of vectors in   $\mathbb R^{n}$; or
   \item[(O)]
    \label{item:O}
       there exist real numbers $q_w$
        for each
        $w\in \mathcal W_N$ not all of which are zero
      such that for each $(X,v)\in\cS$,
      \begin{equation*}
        0= \left(\sum_{|w|\le N} q_w w(X)\right)v.
      \end{equation*}
  \end{itemize}
\end{lemma}

 It is useful to note that the alternative
 Lemma \ref{item:O}(O)   {\bf is equivalent} to
 saying that there exists a (not necessarily symmetric) nc polynomial $q$ of
 degree at most $N$ such that $q(X)v=0$ for every choice of
 $(X, v)\in\cS$.

\begin{proof}
     If condition Lemma \ref{item:E}(E) does not hold, then
     for each positive integer $t$ and each finite set
$F=\{(X^1, v^1),\ldots,(X^t, v^t)\}$ with $(X^j, v^j)\in\mathcal{S}_{n_j}$
for $j=1,\ldots, t$, there is a nonzero function $c_F:\mathcal W_N \to
\mathbb R$
     such  that
   \begin{equation*}
      0= \sum c_F(w) w(X_F)v_F\,.
   \end{equation*}
     Without loss of generality it may be assumed that
     $\sum c_F(w)^2=1$ so that  $c_F$ can be identified
     with an element of $\mathbb B^{L}$, the unit ball in $\mathbb R^L$,
  where $L=\sum_0^N g^j$. For each choice of $F$,
     let $\mathcal C_F\subset \mathbb B^L$
   denote the collection of all such normalized
  coefficients $c_F$ (corresponding to the possibly many
  nc polynomials that annihilate $(X, v)$) and
     observe that each $\mathcal C_F$
     is compact, and of course nonempty by hypothesis.
     Further, if $F\subset G$, i.e., if
$G=\{(X^1, v^1),\ldots,(X^t, v^t), (Y, u)\}$ with
$(Y, u)\in\cS_r$ for some $r$,
then $\mathcal C_F \supset \mathcal C_G$,
     from which it follows that the collection
    \begin{equation*}
       \{\mathcal C_F: \  F \mbox{ is a finite subset of } \cS\}
    \end{equation*}
     satisfies the finite intersection property; i.e., every finite
     intersection is nonempty.
     It follows that the whole intersection is nonempty and
     thus there is a $c\in \mathbb B^L$ so that
    \begin{equation}
     \label{eq:indimn}
       0= \sum c(w)w(X)v \ \ \textrm{for all} \  (X,v)\in\cS.
    \end{equation}
   \end{proof}

\bigskip

 \subsection{Direct Sums of Full Rank  Points}
To show that our main theorem applies to yield
Corollary \ref{cor:sigmain}
we need our full rank assumptions to
mesh with the hypotheses of
Theorem \ref{thm:sigmain}.
The issue is to show that full rank points respect direct sums.

   \begin{lemma}
   \label{lem:DirSmooth}
     Let $p$ be a symmetric nc polynomial in 
     symmetric variables and
     let $F=\{(X^1,v^1),\dots,(X^t,v^t)\}$
     where  $(X^j,v^j) \in\SMATNG{n_j} \times \mathbb R^{n_j}$.
     If each $(X^j,v^j)$ is a full rank point for $p$, then
     so is $(X_F,v_F).$
   \end{lemma}

  \begin{proof}
     Let   $n=n_1+\cdots+n_t$.
     Given $w=\textup{col} (w_1,\ldots, w_t)
\in \mathbb R^n$ with $w_j\in\RR^{n_j}$, 
    there exists $H^j \in \SMATNG{n_j}$ so that
     $p^\prime (X^j)[H^j]v^j=w^j$. This holds because
     each $(X^j,v^j)$ is a full rank point.  Thus, if
$H_F=\textup{diag}\{ H^1,\ldots,H^t\}$, then
     $p^\prime(X_F)[H_F]v_F=w$.
  \end{proof}

    Given an algebraically open set $\cO$ and a symmetric
    nc polynomial $p$, let
   \begin{equation}
     \label{eq:defnBp0}
      \graded{p}{\cO}=\{(X,v): \  (X,v) \ \
      \mbox{is a full rank point}, \ \
          X\in \cO \cap \mathcal V(p)\}.
   \end{equation}

   \begin{lemma}
     \label{lem:exdirsum}
       The set $ \graded{p}{\cO}$
      respects  direct sums.
   \end{lemma}

   \begin{proof} This is an immediate consequence of
     Lemma \ref{lem:DirSmooth} and the fact that both
     $\cO$ and $\mathcal V(p)$ respect direct sums.
  \end{proof}

\section{The Middle Matrix-Border Vector Representation}
\label{sec:middle-border}
  Our approach depends heavily upon the border vector-middle
  matrix representation for non-commutative
  quadratic functions which we now describe.

 A symmetric nc polynomial $f(x)[h]$ in the $2g$ variables
 $x=(x_1,\ldots,x_g)$ and $h=(h_1,\ldots,h_g)$ that is of
 degree $s$ in $x$ and homogeneous of
 degree two in $h$  admits
 a representation of the form
\begin{equation}
 \label{eq:sep2b7}
  f(x)[h]=\begin{bmatrix}V_0(x)[h]^T& \cdots & V_s(x)[h]^T\end{bmatrix}Z(x)
  \begin{bmatrix}V_0(x)[h]\\ \vdots \\ V_s(x)[h]\end{bmatrix},
\end{equation}
  where $Z(x)$ is a square matrix of nc polynomials and
  the $V_j(x)[h]$ are vectors of nc words of the
  form $h_l w(x)$ over choices of words $w$ of length
  $j$.

In the case that $f(x)[h]$ is the Hessian of a symmetric
 nc polynomial $p$ the  middle matrix $Z$ takes a rigid form
 which has been exploited earlier in
 \cite{HM04}, \cite{CHSY03}, \cite{DHMind},
 \cite{DHMjda}
 and \cite{DGHM}:
\begin{equation}
 \label{eq:defs-middle}
   \begin{split}
   p^{\prime\prime}(x)[h]=& V(x)[h]^TZ(x)V(x)[h]  \\
   =& [V^T_0, V^T_1, \ldots, V^T_\ell]
  \left[\begin{array}{ccccc}
  Z_{00}&Z_{01}& \cdots&Z_{0, \ell-1}& Z_{0\ell}\\
  Z_{10}&Z_{11}& \cdots&Z_{1, \ell-1}&0\\
  \vdots&\vdots&&\vdots&\vdots\\
  Z_{\ell 0}&0& \cdots&0&0\end{array}
  \right]
  \left[\begin{array}{c}
  V_0\\V_1\\ \vdots\\ V_{\ell}\end{array}\right],
 \end{split}
\end{equation}
in which $\ell=d-2$, $V(x)[h]$ is the border vector with vector components
$V_j(x)[h]$ of height $g^{j+1}$,
 and $Z(x)=[Z_{ij}(x)]$, $i, j=0,\ldots,d-2$,
the {\bf middle matrix},
is a symmetric matrix  polynomial with matrix polynomial
entries $Z_{ij}(x)$
of size
$g^{i+1}\times g^{j+1}$ and degree no more than $(d-2)-(i+j)$
for $i+j\le d-2$ with
$Z_{ij}(x)=0$ for $i+j> d-2$.
Since $p$ is symmetric $Z_{ij}=Z_{ji}^T$ and since $p$ has
degree $d$,   $Z_{ij}$ is constant when $i+j=d-2$.

The matrix $\mathcal Z=Z(0)$, evaluated at
$0\in\mathbb R^g,$  will be called the {\bf scalar middle matrix} of
$p^{\prime\prime}$.
The main conclusions  from \cite{DHMjda} that are
relevant to this paper are:
\begin{enumerate}
\item[\rm(1)] $Z(x)$ is polynomially congruent to the
   scalar middle  matrix ${\cZ}=Z(0)$, i.e.,
   there exists a matrix polynomial
$B(x)$ with an inverse $B(x)^{-1}$ that is again
a matrix polynomial such that
\begin{equation}
\label{eq:may13a7}
\cZ=Z(0)=B(x)^TZ(x)B(x)\,.
\end{equation}
\vspace{2mm}
\item[\rm(2)] $\mu_{\pm}({\cZ})=\sigma^{min}_\pm(p^{\prime\prime}(x)[h])$.
\vspace{2mm}
\item[\rm(3)] If  $X\in\gtupn$, then
\begin{equation}
\label{eq:apr30a7}
\mu_\pm (Z(X)) =n\mu_\pm(\cZ)\,.
\end{equation}
\item[\rm(4)] The degree $d$ of $p(x)$ is subject to the bound
\begin{equation}
\label{eq:nov10a6}
d\le 2\mu_{\pm}({\cZ})+2\,.
\end{equation}
\end{enumerate}
  Here $\mu_\pm$ are the number of positive/negative eigenvalues of
  the indicated matrix. Note that item \eqref{eq:nov10a6} bounds
  the degree of $p$ in terms of the signature (the number of positive,
  negative and zero eigenvalues) of the middle matrix of its 
  Hessian. 

 The relaxed Hessian $p^{\prime\prime}_{\lambda,\delta}$ also has a
 middle matrix-border vector representation.
 For the special case where $\delta=0$,
 in terms of the  notation introduced in
\eqref{eq:may14a7}, we have
\begin{equation*}
  p^{\prime\prime}_{\lambda,0}(x)[h]
    = \widetilde{V}(x)[h]^T Z_{\lambda}(x) \widetilde{V}(x)[h].
\end{equation*}
The polynomial congruence of equation \eqref{eq:may13a7}
 extends to $Z_\lambda$
 in that
\begin{equation*}
  Z_\lambda(x)  \sim Z_\lambda(0)=: \mathcal{Z}_\lambda =
    \begin{bmatrix} \mathcal Z & 0 \\ 0 & \lambda W\end{bmatrix},
\end{equation*}
 where $W$ is a rank one positive matrix and $\sim$ denotes
 a polynomial congruence which is
 independent of $\lambda$.


\begin{proposition}
 \label{prop:scalar-middle-modified}
   If $X\in\Smatng$, then
\begin{equation*}
     \mu_\pm(Z_\lambda(X)) = n \mu_\pm (\mathcal Z_\lambda).
  \end{equation*}
Moreover, if $\lambda >0$, then
 \begin{equation*}
     \mu_+(\mathcal Z_\lambda) =  \mu_+(\mathcal Z)+1 \quad\textrm{and}\quad
     \mu_-(\mathcal Z_\lambda) = \mu_-(\mathcal Z);
 \end{equation*}
whereas, if $\lambda <0$, then
 \begin{equation*}
     \mu_+(\mathcal Z_\lambda) =  \mu_+(\mathcal Z) \quad\textrm{and}\quad
     \mu_-(\mathcal Z_\lambda) = \mu_-(\mathcal Z)+  1.
 \end{equation*}
\end{proposition}

\begin{proof}
This is an immediate consequence of the polynomial congruence for
$Z_\lambda(x)$ that is described above.
\end{proof}

The middle matrix representation for the relaxed Hessian is
\begin{equation*}
 \label{eq:middle-border-relaxed}
    p^{\prime\prime}_{\lambda,\delta}(x)[h]
      = \widetilde{V}(x)[h]^T Z_{\lambda,\delta}(x) \widetilde{V}(x)[h],
\end{equation*}
where
\begin{equation*}
   Z_{\lambda,\delta}(x)= Z_\lambda(x) + \delta I.
\end{equation*}
  The form of $Z_{\lambda,\delta}$ and the polynomial
  congruence for $Z_\lambda$ together yield the following
  variant of Proposition \ref{prop:scalar-middle-modified}, which is
needed for this paper.

\begin{proposition}
\label{prop:scalar-middle-relaxed}
   Let $X\in \Smatng$ be given.
   There exists an $\epsilon >0$ so that if $0\le \delta <\epsilon$, then
 \begin{equation*}
   \mu_-(Z_{\lambda,\delta}(X)) = n\mu_-(\mathcal Z)\quad\text{for every} \
\lambda\ge 0.
 \end{equation*}
  Similarly, there exists an $\epsilon <0$ so that if
  $\epsilon < \delta \le 0$, then
 \begin{equation*}
   \mu_+(Z_{\lambda,\delta}(X)) = n\mu_+(\mathcal Z)\quad\text{for every}
\ \lambda \le 0.
 \end{equation*}
\end{proposition}

\begin{proof}
The preceding discussion implies that  $Z_{\lambda,\delta}(X)$
is polynomially
congruent to a sum of real symmetric matrices of the form
$$
A+\delta B\quad\textrm{where} \quad A\sim Z_{\lambda,0}(X),\quad
\text{and}\ B\succ 0.
$$
Therefore, if the eigenvalues of each of these matrices are indexed  in
increasing order, i.e., $\lambda_1\le \lambda_2\le\cdots$, it follows
readily from
the Courant-Fischer theorem that if $\delta\ge 0$, then
$$
\lambda_j(A)\le \lambda_j(A+\delta B),
$$
i.e., an additive perturbation of $A$ by  the positive definite matrix
$\delta B$ shifts the eigenvalues of $A$ to the right. Thus, each
nonnegative eigenvalue of $A$ moves into a nonnegative
eigenvalue of $A+\delta B$.
On the other hand, if $\delta\ge 0$ is kept sufficiently
small, so
that  the shift to the right is small, the negative eigenvalues
of $A$ will move into negative eigenvalues of  $A+\delta B$.
Since
$$
\mu_-(A)=\mu_-(Z_{\lambda, 0}(X))=n\mu_-(\cZ)\quad\textrm{for every}\
\lambda \ge 0,
$$
this
completes the proof of the first assertion.
The proof of the second is similar.
\end{proof}




\subsection{Relaxed Hessian example}
 \label{subsec:relaxed-example}
 The example in this subsection illuminates the
middle matrix representation of
 the relaxed Hessian.

\begin{example}\rm
 \label{ex:relaxed}
Let $p(X)=X^3$. Then
$$
p^\prime(X)[H]=X^2H+XHX+HX^2\quad\textrm{and}\quad
p^{\prime\prime}(X)[H]=2HXH+2XH^2+2H^2X.
$$
Therefore,
$$
p^{\prime\prime}_{\lambda,\delta}(X)[H]=
\begin{bmatrix}H\\HX\\HX^2\end{bmatrix}^T
\left\{
\begin{bmatrix}2X&2I&0\\2I&0&0\\0&0&0\end{bmatrix}
+
\lambda\begin{bmatrix}X^2\\X\\I\end{bmatrix}\begin{bmatrix}X^2&X&I\end{bmatrix}
+\delta\begin{bmatrix}I&0&0\\0&I&0\\0&0&I\end{bmatrix}\right\}
\begin{bmatrix}H\\HX\\HX^2\end{bmatrix}.
$$
The middle matrix for the relaxed Hessian is inside the braces.

Moreover, since
$$
\begin{bmatrix}2X&2I&0\\2I&0&0\\0&0&0\end{bmatrix}+
\lambda\begin{bmatrix}X^2\\X\\I\end{bmatrix}
\begin{bmatrix}X^2&X&I\end{bmatrix}
$$
\begin{eqnarray*}
&=&\begin{bmatrix}I&0&X^2\\0&I&X\\0&0&I\end{bmatrix}\begin{bmatrix}I&X/2&0
\\0&I&0\\0&0&I\end{bmatrix}
\begin{bmatrix}0&2I&0\\2I&0&0\\0&0&\lambda I\end{bmatrix}
\begin{bmatrix}I&0&0\\X/2&I&0\\0&0&I\end{bmatrix}
\begin{bmatrix}I&0&0\\0&I&0\\X^2&X&I\end{bmatrix}\\
&=&\begin{bmatrix}I&X/2&X^2\\0&I&X\\0&0&I\end{bmatrix}
\begin{bmatrix}0&2I&0\\2I&0&0\\0&0&\lambda I\end{bmatrix}
\begin{bmatrix}I&0&0\\X/2&I&0\\X^2&X&I\end{bmatrix},
\end{eqnarray*}
we have
$$
p^{\prime\prime}_{\lambda,\delta}(X)[H]=
\begin{bmatrix}H\\HX\\HX^2\end{bmatrix}^T\left\{
\begin{bmatrix}I&X/2&X^2\\0&I&X\\0&0&I\end{bmatrix}
\begin{bmatrix}0&2I&0\\2I&0&0\\0&0&\lambda I\end{bmatrix}
\begin{bmatrix}I&0&0\\X/2&I&0\\X^2&X&I\end{bmatrix}+\delta I\right\}
\begin{bmatrix}H\\HX\\HX^2\end{bmatrix}.
$$

\def\ga{\mathfrak{a}}
\def\gb{\mathfrak{b}}
\def\gc{\mathfrak{c}}

As a more concrete special case,
suppose $X=\textup{diag}\{I_2,-I_3\}$,
$v_1^T=\begin{bmatrix}a&0\end{bmatrix}$,
$v_2^T=\begin{bmatrix}0&0&e\end{bmatrix}$,
$a > e > 0$ and
$v=\textup{col}(v_1,v_2)$.
Then
\begin{multline}
\left\{\begin{bmatrix}H\\HX\\HX^2\end{bmatrix}v:
 \  H\in\RR^{5\times 5}_{sym}
\right\}\\
=\textup{span}\left\{\begin{bmatrix}u_1\\u_1\\u_1\end{bmatrix},\,
\begin{bmatrix}u_2\\u_2\\u_2\end{bmatrix},\,
\begin{bmatrix}u_3\\u_3\\u_3\end{bmatrix},\,
\begin{bmatrix}u_4\\u_4\\u_4\end{bmatrix},\,
\begin{bmatrix}au_5 +eu_1\\au_5-eu_1\\au_5+eu_1\end{bmatrix},\,
\begin{bmatrix}0\\u_2\\0\end{bmatrix},\,
\begin{bmatrix}0\\u_3\\0\end{bmatrix},\,
\begin{bmatrix}0\\u_4\\0\end{bmatrix},\,
\begin{bmatrix}u_5\\-u_5\\u_5\end{bmatrix}\right\},
\end{multline}
where $u_j$ denotes the $j$th standard basis vector for $\RR^5$ for
$j=1,\ldots,5$.
A vector in this span is of the form
$$
w=\begin{bmatrix}\ga_1 +\gb_1\\ \ga_2+ \gb_2\\ \ga_1+ \gb_1\end{bmatrix},
$$
where
$$
\ga_1 =(\alpha_1+\alpha_5e)u_1+\alpha_2u_2, \quad
\ga_2 =(\alpha_1-\alpha_5e)u_1+(\alpha_2+\alpha_6)u_2,
$$
$$
\gb_1=\alpha_3u_3+\alpha_4u_4+(\alpha_5a+\alpha_9)u_5\quad\textrm{and}\quad
\gb_2=(\alpha_3+\alpha_7)u_3+(\alpha_4+\alpha_8)u_4+(\alpha_5a-\alpha_9)u_5
$$
for some choice of $\alpha_1,\ldots,\alpha_9\in\RR$.

Next,
since $X\ga_j =\ga_j$, $X\gb_j =-\gb_j$ for $j=1,2$, it is readily seen that
$$
\begin{bmatrix}I&0&0\\X/2&I&0\\I&X&I\end{bmatrix}
\begin{bmatrix} \ga_1+\gb_1\\ \ga_2+\gb_2\\ \ga_1+\gb_1\end{bmatrix}=
\begin{bmatrix} \ga_1+\gb_1\\ \ga_3+\gb_3\\ \ga_4+\gb_4\end{bmatrix},
$$
where
$$
\ga_3=\ga_2+(1/2)\ga_1,\, \gb_3=\gb_2-(1/2)\gb_1,\, \ga_4=2\ga_1+\ga_2,\
\textrm{and}\
\gb_4=2\gb_1-\gb_2.
$$
Thus, as
$$
\ga_i^T\gb_j=0\quad\textrm{for}\ i,j=1,\ldots,4,
$$
it is readily seen that
$
v^Tp^{\prime\prime}_{\lambda,\delta}(X)[H]v$ is of the form
\begin{eqnarray*}
&{}&w^T\left\{
\begin{bmatrix}I&X/2&X^2\\0&I&X\\0&0&I\end{bmatrix}
\begin{bmatrix}0&2I&0\\2I&0&0\\0&0&\lambda I\end{bmatrix}
\begin{bmatrix}I&0&0\\X/2&I&0\\X^2&X&I\end{bmatrix}+\delta I\right\}w\\
&=&
\ga_1^T\ga_3+\ga_3^T\ga_1+\gb_3^T\gb_1+\gb_1^T\gb_3+\lambda(\ga_4^T\ga_4+
\gb_4^T\gb_4)\\
&{}&  +\delta(2\ga_1^T\ga_1+\ga_2^T\ga_2+2\gb_1^T\gb_1+\gb_2^T\gb_2).
\end{eqnarray*}
In particular, the term to the right of $\delta$ is equal to zero if and
only if
$$\alpha_1=\cdots=\alpha_9=0.$$
\end{example}

\begin{remark}\rm
We remark that since
$$
Q:=\begin{bmatrix}I&X/2&X^2\\0&I&X\\0&0&I\end{bmatrix}^{-1}=
\begin{bmatrix}I&-X/2&-X^2/2\\0&I&-X\\0&0&I\end{bmatrix},
$$
the formula for the relaxed Hessian can be re-expressed as
\[
\begin{split}
  p^{\prime\prime}_{\lambda,\delta}(X)[H]=
\begin{bmatrix}H\\HX\\HX^2\end{bmatrix}^T
\begin{bmatrix}I&X/2&X^2\\0&I&X\\0&0&I\end{bmatrix}\left\{
\begin{bmatrix}0&2I&0\\2I&0&0\\0&0&\lambda I\end{bmatrix}+\delta QQ^T\right\}
\times \\
 \times \begin{bmatrix}I&0&0\\ X/2&I&0\\X^2&X&I\end{bmatrix}
\begin{bmatrix}H\\HX\\HX^2\end{bmatrix}.
\end{split}
\]
\end{remark}


\section{The CHSY Lemma}
 \label{sec:CHSY}
 The CHSY Lemma (which is based on Lemma 9.5 in \cite{CHSY03}) is the key tool relating
 the signature of the middle matrix of a quadratic form $q$ with the
 signature of $q$.
 In this section we develop a version of this lemma that is  required
 for the proof of Theorem \ref{thm:sigmain}.

 Let
 $(X, v)\in \Smatng\times\RR^n$, let $H\in \Smatng\times\RR^n$ and let
\begin{equation}
 \label{eq:sep2a7}
  {\mathcal R}_s(H)=\begin{bmatrix}V_0(X)[H]v\\ \vdots \\ V_s(X)[H]v
  \end{bmatrix}\quad\textrm{and}\quad \mathcal{R}_s(\cH)=\{\mathcal{R}_s(H):\, H\in{\cH}\}\,,
\end{equation}
 where  the vectors $V_j(x)[h]$
are defined just above (\ref{eq:may14a7}), ${\cH}$ is a subspace of $\Smatng$ and 
 $\mathcal R_s(\mathcal H)$ is a subspace of
 $\mathbb R^{ng\alpha_s}$.

\begin{lemma}[CHSY Lemma]
 \label{lem:CHSY}
  Given a pair of positive integers $g$ and $r$, a matrix $X\in\gtupn$
  and a vector $v\in \mathbb R^{n}$, suppose that the set 
  $$
  \{w(X)v:\ w\  \textrm{is a word with}\ \vert w\vert\le r\}
  $$
   is a linearly independent
  subset of $\mathbb R^n$. Then $\mathcal{R}_s(\gtupn)$ is a subspace of
  $\RR^{ng\alpha_s}$ and
\begin{equation}
 \label{eq:aug22a7}
   \textup{codim}\,\mathcal{R}_s(\gtupn)\le ng(\alpha_s-\alpha_r)
+g\alpha_r
   \frac{\alpha_r-1}{2}
   \quad\text{if}\quad s\ge r.
\end{equation}

  If $s=r$, then the codimension of $\mathcal R_r(\gtupn)$
  is independent of $n$ and
\begin{equation}
 \label{eq:sequalsr}
    \textup{codim}\, \mathcal R_r(\gtupn) =  g\alpha_r
\frac{\alpha_r-1}{2}.
\end{equation}
\end{lemma}

\begin{remark}
 \label{rem:sep12a7}
  It is important to bear in mind that if $\cH$ is a subspace of
$\Smatng$, then
$$
 \textup{codim}\,(\cH)=g\frac{n(n+1)}{2}-\textup{dim}\,(\cH)\,,
$$
whereas
$$
 \textup{codim}\,(\mathcal R_s(\cH))
 =ng\alpha_s-\textup{dim}\,(\mathcal R_s(\cH))
 \,.
$$
\end{remark}

\begin{lemma}
 \label{lem:aug30a7}
  If  ${\cH}$ and ${\cH}^c$ are
  complementary subspaces of $\Smatng$, i.e., if
$$
\Smatng={\cH}\dot{+}{\cH}^c\,,
$$
  then ${\mathcal R}_s(\cH)$ is a subspace
  of $\RR^{ng\alpha_s}$ and
\begin{equation*}
 \begin{split}
   \textup{codim}\,{\mathcal R}_s(\cH)\le &  \textup{codim}\,{\mathcal R}_s(\Smatng)
     +\textup{codim}\,\cH \\
   \textup{codim}\,{\mathcal R}_s(\cH^c)\le &  \textup{codim}\,{\mathcal R}_s(\Smatng)
     +\textup{dim}\,\cH \, .
 \end{split}
\end{equation*}
\end{lemma}

\begin{proof}
If ${\cH}^c$ is a complementary subspace to $\cH$ in $\Smatng$, then
\begin{eqnarray*}
\textup{dim}\,{\mathcal R}_s(\Smatng)&=&\textup{dim}\,{\mathcal R}_s(\cH)+
\textup{dim}\,{\mathcal R}_s(\cH^c)\\
&\le &\textup{dim}\,{\mathcal R}_s(\cH)+\textup{dim}\,\cH^c\,.
\end{eqnarray*}
 The first asserted inequality now follows easily upon re-expressing the last
 inequality in terms of codimensions.

  The second asserted inequality follows from the first by replacing $\mathcal H$
  by $\mathcal H^c$ and noting that the dimension of $\mathcal H$ is equal
  the codimension of $\mathcal H^c$.
\end{proof}

\begin{lemma}
\label{lem:nov25a7}
Let $Z(x)$ denote the middle matrix in the representation (\ref{eq:sep2b7})
of a symmetric nc polynomial $f(x)[h]$ of degree $s$ in the $g$
symmetric
variables $x=(x_1,\ldots,x_g)$ that is
homogeneous of
degree two in the $g$ symmetric variables $h=(h_1,\ldots,h_g)$, and
let $\cH$  and $\cG$ denote subspaces of
$\Smatng$ such that
$$
\langle f(X)[H]v, v\rangle< 0\quad\text{if $H\in\cH$ and $H\ne 0$}
\quad \text{and}\quad  \langle f(X)[H]v, v\rangle\ge 0 \quad\text{if}
\ H\in\cG.
$$
Then
\begin{equation}
\label{eq:nov25b7}
\textup{codim}\,{\cR}_s(\cG)\ge
\mu_-(Z(X))\ge \textup{dim}\,\cR_s(\cH)\,,
\end{equation}
\begin{equation}
\label{eq:nov25c7}
\textup{codim}\,{\cR}_s(\cH)\ge
\mu_+(Z(X))+\mu_0(Z(X))\ge \textup{dim}\,\cR_s(\cG)
\end{equation}
and
\begin{equation}
\label{eq:dec28a7}
\textup{dim}\,\cR_s(\cH)\ge \textup{dim}\,\cH\,.
\end{equation}
\end{lemma}

\begin{proof}
The lower bounds in (\ref{eq:nov25b7}) and (\ref{eq:nov25c7})
are self-evident. The upper bounds then follow from the identities
\begin{eqnarray*}
\mu_-(Z(X))+\mu_0(Z(X))+\mu_+(Z(X))&=&\textup{dim}\,{\cR}_s(\cH)+
\textup{codim}\,{\cR}_s(\cH)\\
&=&\textup{dim}\,{\cR}_s(\cG)+
\textup{codim}\,{\cR}_s(\cG),
\end{eqnarray*}
upon re-expressing the two
lower bounds in terms of codimensions.

To verify (\ref{eq:dec28a7}), it suffices to note that if
$H_1,\ldots,H_k$ is a basis for $\cH$, then the vectors
$\cR_s(H_1),\ldots,\cR_s(H_k)$ must be linearly independent because of the
presumed strict negativity of $\cH$.
\end{proof}

\begin{lemma}
\label{lem:dec28a7}
If the subspaces $\cG$ and $\cH$ considered in Lemma \ref{lem:nov25a7}
are such that $\cG=\cH^c$ is complementary to $\cH$, then
\begin{equation}
 \label{eq:aug30a7}
   \textup{codim}\,{\mathcal R}_s(\cH^c)\ge
   \mu_-(Z(X)) \ge \textup{dim}\, \mathcal H
\end{equation}
  and
\begin{equation}
 \label{eq:aug30b7}
   \textup{dim}\,{\mathcal R}_s(\cH^c)\le
   \mu_+(Z(X))+\mu_0(Z(X)) \le ng\alpha_s -\textup{dim}\, \mathcal H.
\end{equation}
\end{lemma}

\begin{proof}
This is an immediate consequence of Lemma \ref{lem:nov25a7}.
\end{proof}

\begin{lemma}
\label{lem:aug26a8}
Let $A\in\RR_{sym}^{n\times n}$ and let $\cU$ be a maximal strictly negative
subspace of $\RR^n$ with respect to the quadratic form
$\langle Au, u\rangle$. Then there exists a complementary subspace $\cV$ of
$\RR^n$ such that $\langle Av, v\rangle\ge 0$ for every $v\in\cV$.
\end{lemma}

\begin{proof}
Let $U\in\RR^{n\times n}$ be an orthogonal matrix and $D\in\RR^{n\times n}$
be a diagonal matrix such that $AU=UD$ and assume that $\mu_-(A)=k_1>0$,
$\mu_+(A)=k_2>0$ and $\mu_0(A)=k_3>0$. Then we may assume that
$$
U=[U_1\ U_2\ U_3]\quad\text{and}\quad D=\textup{diagonal}\{D_1,D_2,D_3\},
$$
where
$$
AU_i=U_iD_i\quad\text{for}\ i=1,2,3, \quad
D_1\prec 0,\quad D_2\succ 0\quad \textrm{and}\quad D_3=0.
$$
Now let $u_1,\ldots, u_{k_1}$ be a basis for $\cU$.
Then, since the columns
of $U$ span $\RR^n$, there exists a matrix $M\in\RR^{n\times k_1}$ with
blocks $M_{i1}\in\RR^{k_i\times k_1}$ for $i=1,2,3$ such that
$$
[u_1\ \cdots \ u_{k_1}]=UM=U_1M_{11}+U_2M_{21}+U_3M_{31}.
$$
The next step is to check that $M_{11}$ is invertible. But, if $M_{11}c=0$
for some vector $c\in\RR^{k_1}$ with components $c_1,\ldots,c_{k_1}$, then
\begin{eqnarray*}
\left\langle A\sum c_iu_i, \sum c_iu_i\right\rangle
&=& \langle AUMc, UMc\rangle\\
&=& \langle U_2D_2M_{21}c+U_3D_3M_{31}c,U_2M_{21}c+U_3M_{31}c \rangle\\
&=& \langle D_2M_{21}c, M_{21}c\rangle\ge 0.
\end{eqnarray*}
Therefore, since $\cU$ is a strictly negative subspace, it follows that
$c=0$. Thus, $M_{11}$ is invertible. Let $M_{ij}\in\RR^{k_i\times k_j}$
for $i=2, 3$ and $j=1,2,3$ with $M_{22}$ and $M_{33}$ invertible, and let
$$
\cV=\textup{span}\left\{U\begin{bmatrix}0&0\\ M_{22}&0\\M_{32}&M_{33}
\end{bmatrix}b:\,b\in\RR^{k_2+k_3}\right\}.
$$
Then it is readily checked that $\cV$ is a complementary subspace to
$\cU$ in $\RR^n$ and that $\langle Av,v\rangle\ge 0$ for every $v\in\cV$.
\end{proof}

 The following proposition ties this section in with Section 
 \ref{sec:middle-border}, by relating the
 number of negative  eigenvalues of
 $Z(X)$ to the dimension of a maximal negative
 subspace of the clamped second fundamental form.

\begin{proposition}
 \label{prop:chsy-applied}
  Let $\mathcal Z$ be the scalar middle matrix of the Hessian
$p^{\prime\prime}$ of
  a symmetric nc polynomial $p$ 
  in symmetric variables and  let
  $(X,v)\in\Smatng \times \mathbb R^n.$
  There is an $\epsilon>0$ such that if
  $0<\delta <\epsilon$,   $\lambda>0$ and $\mathcal H$ is a
  maximal strictly negative subspace
  for the quadratic form
\begin{equation*}
(\RR_{sym}^{n\times n})^g \ni H\mapsto
  \langle p^{\prime\prime}_{\lambda,\delta}(X)[H]v,v\rangle
\end{equation*}
(based on the relaxed Hessian), then
\begin{equation}
\label{eq:jul26a9}
      \textup{dim}\, \mathcal H \le n\mu_-(\mathcal Z)
          \le \textup{dim}\, \mathcal H + \textup{codim}\,
     \mathcal R_{d-1}(\Smatng).
 \end{equation}

\end{proposition}

\begin{proof}
  Let $d$ denote the degree of $p$ and $g$ the number of variables. 
  Choose $\epsilon >0$ as in Proposition \ref{prop:scalar-middle-relaxed}.
  Then for $0<\delta <\epsilon$ and $\lambda >0$
 \begin{equation}
  \label{eq:chsy-1}
     \mu_-(Z_{\lambda,\delta}(X)) = n\mu_-(\mathcal Z).
 \end{equation}

  Since $\cH$ is a maximal strictly negative subspace, the space
  $\cG$ considered
  in Lemma \ref{lem:nov25a7} can be chosen to coincide
  with a complementary subspace $\cH^c$
  to $\cH$ (thanks to Lemma \ref{lem:aug26a8}) and hence, by
Lemmas \ref{lem:dec28a7} and \ref{lem:aug30a7},
\begin{eqnarray*}
  \textup{dim}\, \mathcal H &\le&
  \mu_-(Z_{\lambda,\delta}(X)) \le
\textup{codim}\,{\mathcal R}_{d-1}(\cH^c) \\
&=&\textup{codim}\,\mathcal R_{d-1}(\Smatng)
+\textup{dim}\,{\cR}_{d-1}({\cH})\\
&\le& \textup{dim}\, \mathcal H + \textup{codim}\,
\mathcal R_{d-1}(\Smatng).
\end{eqnarray*}
The rest follows from (\ref{eq:chsy-1}).
\end{proof}

\begin{remark}\rm
 \label{rem:delta}
If the relaxed Hessian is negative definite, then the lower bound in
(\ref{eq:jul26a9}) applied to $\cH=\Smatng$ implies that
\begin{equation*}
\textup{dim}\,\Smatng=\frac{n(n+1)}{2}g
\le n\mu_-(\mathcal Z).
 \end{equation*}
Therefore, since
$$
\mu_-(\cZ)\le g+g^2+\cdots+g^{d-1},
$$
the relaxed Hessian cannot be negative definite
if $n>2(1+g+\cdots+g^{d-2})-1$.
\end{remark}


\section{Proof of Theorem \ref{thm:sigmain} and related results}
 \label{sec:pfofsigmain}
   In this section we prove 
  Theorem \ref{thm:sigmain}, Corollary \ref{cor:sigmain} and 
  some variations thereof.
  The first subsection
  contains a proof of the
  existence of the limit $C_\pm(\mathcal S)$; the
  second verifies the inequality in
  \eqref{eq:dec31a7}; the remaining parts of the theorem
  and Corollary \ref{cor:sigmain}
  are proved in the third subsection. 
  Some supplementary results 
are given in the fourth and final subsection.

\subsection{The existence of $C_\pm$}
 \label{subsec:C-exists}
We shall need the following result:
\begin{lemma}
 \label{lem:Hinequality}
   Suppose $p$ is a symmetric nc polynomial 
   in symmetric  variables. Let $(X,v)\in \Smatng\times \mathbb R^n$ be  given.
If $k$ is a positive integer and
\begin{equation*}
  Y=\textup{diag}\{X,\ldots,X\}\quad\textrm{and}\quad
  w=\textup{col}(v,\ldots,v) \quad k\ \textup{times},
\end{equation*}
then
$$
  c_\pm^{kn}(Y,w;p) \ge k c_\pm^{n}(X,v;p).
$$
\end{lemma}

\begin{proof}
  Recall that $c_+^{n}(X,v;p)$ is the maximum dimension of a strictly
positive
  subspace of $\mathcal T^n\subset \Smatng$ with respect to the quadratic
 form,
 \begin{equation*}
  \langle H, K \rangle = \langle p^{\prime\prime}(X)[H][K]v,v\rangle,
 \end{equation*}
  where
 \begin{equation*}
   \mathcal T^n=\{H\in\Smatng: p^\prime(X)[H]v=0\}
 \end{equation*}
and the superscript $n$ has been added because the size of the matrices
under consideration is now an issue.
   Similarly, $c_+^{kn}(Y,w;p)$ is the dimension of a maximal positive
   subspace of $\mathcal T^{nk}$ relative to the form
 \begin{equation}
  \label{eq:formnk}
  \langle H, K \rangle = \langle p^{\prime\prime}(Y)[H][K]w,w\rangle,
 \end{equation}
  where
 \begin{equation*}
   \mathcal T^{kn}=\{H\in\mathbb (R^{nk\times nk}_{sym})^g : \
p^\prime(Y)[H]w=0\}.
 \end{equation*}
   Let $\mathcal P_n$ denote a positive subspace of $\Smatng$ with
   $\textup{dim}\,\cP_n=c_+^{n}(X,v;p)$ and let
 \begin{equation*}
   \mathcal Q_{nk} =\{ \textup{diag}\{H^1,\ldots,H^k\} : \  H^j\in
\mathcal P_n\}.
 \end{equation*}
  Then $\mathcal Q_{nk}\subseteq \mathcal T^{nk}$ and $\mathcal Q_{nk}$
  is positive relative to the form in equation \eqref{eq:formnk}.
  Therefore,
$$
c_+^{kn}(Y,w;p)\ge kc_+^n(X,v;p),
$$
since the dimension of $\mathcal Q_{nk}$ is $k$ times the
dimension of $\mathcal P_n$. The verification of the analogous inequality
with $-$ instead of $+$ is similar.
\end{proof}

\begin{proof}[Proof of (i) in Theorem \ref{thm:sigmain}]
The bound
$$
\frac{c_\pm^n(X,v;p)}{n} \le \mu_\pm(\mathcal Z)+\mu_0(\mathcal Z)
$$
guarantees that
 \begin{equation*}
   \Gamma_\pm = \sup_n \left(\sup
\left\{\frac{c_\pm^n(X,v;p)}{n}: \  (X,v)\in\mathcal S_n\right\}\right)
 \end{equation*}
is finite.

Let
\begin{equation}
\label{eq:aug29a8}
 \beta^n_\pm = \sup
\left\{\frac{c_\pm^n(X,v;p)}{n}: \  (X,v)\in\mathcal S_n\right\}.
\end{equation}
We shall prove that
$
\beta_\pm^n
\rightarrow \Gamma_\pm \quad\textrm{as}\quad n\uparrow\infty.
$
by showing that given any $\varepsilon >0$,
there exists an $N>0$ such that
$$
\Gamma_\pm \ge \beta_\pm^n\ge \Gamma_\pm-\varepsilon
$$
for every integer $n\ge N$. Since the upper bound
$\Gamma_\pm \ge \beta_\pm^n$ is clear, it suffices to verify the lower bound
when $\Gamma_\pm-\varepsilon>0$.
Under this assumption, there exists a positive integer $t$ and a pair
$(X,v)\in (\RR^{t\times t}_{sym})^g\times \RR^t$ such that
$$
\frac{c_\pm^t(X,v;p)}{t} \ge \Gamma_\pm-\varepsilon/2.
$$
Let $k_0\ge \frac{2\Gamma_\pm}{\varepsilon}.$
Then for any integer $n> k_0t$, there exists an integer $k\ge k_0$
such that
$$
kt< n\le (k+1)t.
$$
 Here we use the hypothesis that $\cS_1,$ and
 hence $\cS_m$ for every $m$,  is nonempty. 
Since, by hypothesis, $\cS_{n-kt}$ is nonempty, there
is a pair $(Z,u)\in \cS_{n-kt}$.
Let
$(Y,w)=\oplus_1^k (X,v) \oplus (Z,u)$
Then, since $Y\in(\RR^{n\times n}_{sym})^g$ and $w\in\RR^n$,
$$
c_\pm^n(Y,w;p) \ge kc_\pm^t(X,v;p)+ c_\pm^n(Z,u;p)
\ge kc_\pm^t(X,v;p).
$$

Therefore,
\begin{eqnarray*}
\Gamma_\pm&\ge& \beta_\pm^n\ge \frac{c_\pm^n(Y,w;p)}{n}
\ge
\frac{kc_\pm^t(X,v;p)}{n}
\ge \frac{kc_\pm^t(X,v;p)}{(k+1)t}\\
&\ge&\frac{k}{(k+1)}\frac{c_\pm^t(X,v;p)}{t}\\
&\ge&\frac{k_0}{(k_0+1)}(\Gamma_\pm-\varepsilon/2)
\quad(\textrm{since}\ k\ge k_0) \\
&\ge&\frac{2\Gamma_\pm/ \varepsilon}{1+2\Gamma_\pm/ \varepsilon}
(\Gamma_\pm-\varepsilon/2)\quad(\textrm{since}\ k_0\ge 2\Gamma_\pm/
\varepsilon)\\
&\ge&\Gamma_\pm-\varepsilon.
\end{eqnarray*}
\end{proof}

\subsection{Proof of inequality \eqref{eq:dec31a7}}
 \label{subsec:proof-ineq-thm-main}
 The representation formula (\ref{eq:defs-middle}) for the Hessian
 $p^{\prime\prime}(x)[h]$ of $p$ and formulas (\ref{eq:may13a7}) and
 (\ref{eq:apr30a7}) imply that
\begin{equation}
 \label{eq:nmue}
   n\mu_-(\mathcal Z) \ge e_-^n(X,v;p^{\prime\prime},\gtupn)\ge
   e_-^n(X,v;p^{\prime\prime},\cT)
  =c_-^n(X,v;p)
\end{equation} 
and hence, in view of \eqref{eq:aug29a8} that
\begin{equation*}
   \beta_-^n \le \mu_-(\mathcal Z).
\end{equation*}

The next lemma is a step in the proof of (\ref{eq:dec31a7}) that contains 
information of independent interest.

\begin{lemma}
 \label{lem:localq-plus}
   Let $p$ be a symmetric nc polynomial 
   in symmetric variables and let $(X,v)\in \Smatng \times \mathbb R^n$
   be given.
  If
  there is no nonzero nc polynomial $q$ (not necessarily symmetric)
  of degree less than $d$ such that $q(X)v=0$,
  then
 \begin{equation}
 \label{eq:cnmucn}
    \frac{c_\pm^n(X,v;p)}{n}\le \mu_\pm(\mathcal Z) \le
    \frac{g\alpha_{d-1}(\alpha_{d-1}-1)}{2n}+
    \frac{c_\pm^n(X,v;p)}{n}.
 \end{equation}

   In particular, if $c_-^n(X,v;p)=0$ and
  $2n> g\alpha_{d-1}(\alpha_{d-1}-1)$, then
  $\mu_-(\mathcal Z)=0$.
\end{lemma}


\begin{proof}
  Let $d$ denote the degree of $p$ and $g$ the number of variables.
  By Lemma \ref{lem:CHSY} with $r=s=d-1$,
\begin{equation*}
  \textup{codim}\,\mathcal R_{d-1}(\Smatng)=
  g\alpha_{d-1}\frac{\alpha_{d-1}-1}{2}.
\end{equation*}
  By Theorem \ref{thm:signature-clamped-relaxed},
  there is a $\delta_0>0$ such that for
  each $0<\delta\le \delta_0$ there exists
  a $\lambda>0$ such that
 $$
  e^n_-(X,v;p^{\prime\prime}_{\delta,\lambda},\Smatng)
     =c^n_-(X,v;p).
 $$
  With $\epsilon>0$ as in Proposition \ref{prop:scalar-middle-relaxed}
  and $0<\delta<\epsilon$ (as well as $\delta<\delta_0$),
  the second inequality in Proposition \ref{prop:chsy-applied}
  implies that
 \begin{equation}
  \label{eq:proof-sigmain-2}
    n \mu_-(\mathcal Z) \le c_-^n(X,v;p)
    +  g\alpha_{d-1}(\alpha_{d-1}-1)/2 .
 \end{equation}
Thus, in view of \eqref{eq:nmue}, 
the inequalities in (\ref{eq:cnmucn}) hold for the numbers $c_-^n(X,v;p)$ 
and $\mu_-(\cZ)$. However, since 
\begin{equation}
\label{eq:sep21a10}
c_\pm^n(X,v;p)=c_\mp^n(X,v;-p)\quad\textrm{and} \quad 
\mu_\pm(\cZ)=\mu_\mp(-\cZ),
\end{equation}
these  inequalities also hold for the numbers $c_+^n(X,v;p)$ 
and $\mu_+(\cZ)$.
\end{proof}

   Returning to the proof of inequality \eqref{eq:dec31a7}, 
   since $\mathcal S$ respects direct sums and $p$ is a minimum
   degree defining polynomial
   for $\cS$, the condition Lemma \ref{item:E}(O)
   with $N=d-1$ can not hold. Thus,
   Lemma \ref{lem:directsumsystem} guarantees that
   there is an integer $j$ and a pair
   $(Y,w)\in\mathcal S_j$ such that  $\{m(Y)w: |m|<d\}$ is linearly
   independent.
   For a fixed positive integer $k$, let
   $(X,v)=\oplus_1^k (Y,w)$.
   Then $(X,v)\in\mathcal S_{jk}$
   and $\{m(X)v: \  |m|<d\}$ is linearly independent.
   This linear independence is equivalent to
   the hypothesis of Lemma \ref{lem:localq-plus} and   hence,
   in view of \eqref{eq:nmue},
\beq
\label{eq:cnmucnn}
 \mu_-(\mathcal Z) \le
\frac{g\alpha_{d-1}(\alpha_{d-1}-1)}{2n}+
\frac{c_-^n(X,v;p)}{n}.
\end{equation}
Consequently,
$$
\beta_-^n\le \mu_-(\cZ)\le \frac{g\alpha_{d-1}(\alpha_{d-1}-1)}{2n}+
\beta_-^n
$$
and hence
$$
\limsup _{n\uparrow \infty} \beta_-^n\le \mu_-(\cZ)\le
\liminf_{n\uparrow\infty} \beta_-^n.
$$
Therefore,
 \begin{equation}
 \label{eq:aug29b8}
   C_-(\cS)=\lim_{n\uparrow\infty}\beta_-^n=\mu_-(\cZ).
\end{equation}

By a similar argument, or 
by exploiting (\ref{eq:sep21a10}) and \eqref{eq:aug29b8},
\begin{equation}
\label{eq:nov26a8}
C_+(\cS)=\lim_{n\uparrow\infty}\beta_+^n=\mu_+(\cZ).
\end{equation}
The  bounds (\ref{eq:dec31a7}) follow easily from the identifications
(\ref{eq:aug29b8}) and (\ref{eq:nov26a8})
and the bounds (\ref{eq:nov10a6}).

\begin{remark}
If $p$ is a $k$-minimum degree defining polynomial,
then the argument proving \eqref{eq:cnmucnn} can be
modified as follows:

 $\{w(Y)u: \  |w|<d-k\}$ is linearly
   independent. For a fixed positive integer $\ell$, let
   $(X,v)=\oplus_1^\ell (Y,u)$. Then $(X,v)\in\mathcal S_{j\ell}$
   and $\{w(X)v: \  |w|<d-k\}$ is linearly independent.
   By Lemma \ref{lem:CHSY} with
   $n=j\ell$ and $s=d-1$ and $r=d-1-k$, the codimension
   of $\mathcal R_{d-1}(\Smatng)$ is at most
   $$
ng(\alpha_{d-1}-\alpha_{d-1-k})+g\alpha_{d-1-k}(\alpha_{d-1-k}-1)/2.
$$
Hence, by
   the second inequality in Proposition \ref{prop:chsy-applied},
 \begin{equation}
  \label{eq:proof-sigmain-2a}
    n \mu_-(\mathcal Z) - ng(\alpha_{d-1}-\alpha_{d-1-k}) -
\frac{g\alpha_{d-1-k}(\alpha_{d-1-k}-1)}{2}
\le c_-(X,v;p).
 \end{equation}
\end{remark}

\subsection{Proof of (B)--(D) in Theorem \ref{thm:sigmain}
and Corollary \ref{cor:sigmain}}
 Returning to the assumption that $p$ is a
 minimum degree defining polynomial for
 $\cS$, if $C_-(\cS)=0$, then, by
 equation \eqref{eq:aug29b8}, $\mu_-(\mathcal Z)=0$; i.e.,
 the middle matrix
 for $p^{\prime\prime}(x)[h]$ is
  positive (resp., negative) semi-definite
  and thus also constant (i.e., $Z(x)=Z(0)=\cZ$).
  The factorization of
 $Z(x)=\mathcal Z$ as $W^*W$ in the positive semi-definite case
 shows that
  $p(x)= L(x) + \Lambda(x)^T \Lambda(x)$, where $L(x)$ has degree at most one
 and $\Lambda(x)$ is either equal to zero or to a homogeneous polynomial of
 degree one. In particular, $p$ is convex.

 A similar argument prevails in the case that $C_+(\cS)=0$.

  For the converse, if $p$ is convex, then
 \begin{equation*}
     p^{\prime\prime}(X)[H]
 \end{equation*}
  is positive semi-definite for all $X,H$ and thus
 $c_-(X,v;p)=0$ (for all $X$ and $v$).

 Note that the verification of the  equalities $C_\pm(\cS)=\mu_\pm(\cZ)$
 from equations \eqref{eq:nov26a8} and \eqref{eq:aug29b8} 
  in the proof of Lemma \ref{lem:localq-plus} depends only upon
 $S$ being a nonempty 
  set which
 is  closed with respect to direct sums and for which
 $p$ is a minimum  degree defining polynomial.   Consequently
 Theorem \ref{thm:sigmain} (D) holds.

\begin{proof}[Proof of Corollary \ref{cor:sigmain}]
  The set of full rank points of $p$
  in $\mathcal V(p) \cap \mathcal O$ respects
  direct sums (see Lemma \ref{lem:DirSmooth}).
  If $p$ has positive
  curvature on $\cS$, then for each
$(X,v)\in\cS_n$, we have
  $c_+(X,v;p)=0$ and hence $C_+(\cS)=0$.
  The conclusion
  now  follows from statement (B) in  Theorem \ref{thm:sigmain}.
\end{proof}

\subsection{Determining sets for the signature of $\cV(p)$}

This subsection
  explores and expands upon  the principle
  (mentioned in the introduction, 
  see Proposition \ref{prop:localq-simple})
   that the signature of $\cV(p)$ 
  is determined on any subset
  $\cS$ respecting direct sums, which is large enough so that $p$ is a minimal defining
  polynomial for $\cS$.  

\begin{theorem}
 \label{prop:localq-plus}
  Let $p$ be a symmetric nc polynomial  of degree $d$ in $g$ 
  symmetric variables.
  Suppose
\begin{equation}
\label{eq:may13c7}
  n>\frac{1}{2}g
  \alpha_{d-1}(\alpha_{d-1}-1),
\end{equation}
  $X\in\gtupn,$  $v\in\mathbb R^n,$ and that
\begin{enumerate}
\item[(a)] $(X,v)\in\cV(p)$; and
\item[(b)] there is no nonzero nc polynomial $q$ (not necessarily symmetric)
  of degree less than $d$ such that $q(X)v=0$; and
\item[(c)] $\cS$ is any 
  subset of 
  $\cV(p)$ 
  which is nonempty, closed with respect to direct sums and for which
  $p$ is a minimum degree defining polynomial (so that hypotheses of
  Theorem \ref{thm:sigmain} (A)(i) and (ii) hold), 
\end{enumerate}
  then 
\begin{equation*}
C_\pm(\cS)   = \lceil \frac{c_\pm^n(X,v;p)}{n} \rceil
\end{equation*}
were $\lceil r\rceil$ is the ceiling function; i.e.,
the the smallest integer 
bigger than or equal to $r$.  
\end{theorem}

\begin{proof}[Proof of Theorem \ref{prop:localq-plus}]

 This follows from the inequalities (\ref{eq:cnmucn}), formulas 
(\ref{eq:aug29b8}) and (\ref{eq:nov26a8}) (which also serve to identify 
$C_\pm(\cS)$ as integers) and the bound (\ref{eq:may13c7}), which insures 
that 
\begin{equation*}
   1> \frac{1}{2n}g\alpha_{d-1}(\alpha_{d-1}-1).
  \end{equation*}
\end{proof}

\begin{proof}[Proof of Proposition \ref{prop:localq-simple}]
  The first part of the proposition is an immediate consequence of
  Theorem \ref{prop:localq-plus}.  The existence of a pair $(X,v)$ 
  with the properties claimed in the second part of the proposition 
  follows from the the second half of the proof of 
  \eqref{eq:dec31a7} in Section  \ref{subsec:proof-ineq-thm-main}.
\end{proof} 


\section{Positive Curvature and Convex Sets}
 \label{sec:convex-sets}
If $p=p(x_1,\ldots,x_g)$ is a concave polynomial
in $g$ commuting variables,
then for each of $\alpha\in\RR$, the superlevel set
$$
L_\alpha=\{x:\,p(x)\ge \alpha\}
$$
is either empty or a convex subset of $\RR^g$. The converse
is false. However, in the classical setting,
a defining polynomial for a convex set $\cC$ with smooth boundary
has second derivative which, when
restricted to the tangent plane $T\cC$ at each point of $\partial{\cC}$
is negative semi-definite, or, 
in the language of
differential geometry, this is the same as to say that
the second fundamental form is positive semi-definite.
This section discusses a non-commutative analog.

  Let $p$ denote a symmetric nc polynomial of degree $d$
  in $g$ symmetric variables. Assume that $p(0)\succ 0$;
  i.e., the constant term of $p$ is strictly positive.
  The {\bf positivity domain} \index{positivity domain}
   of such a  $p$ in dimension
  $n$, denoted $\posdom{p}^n$, is the closure of
  the component of $0$ of the set
 \begin{equation*}
  \cP_p^n = \{ X \in\Smatng: p(X)\succ 0\}.
 \end{equation*}
  As usual, let $\posdom{p} =\cup_n \posdom{p}^n.$
  Let $\boundary{p}^n$ denote the
  boundary of the set $\posdom{p}^n$ and
  $\boundary{p}= \cup_n \boundary{p}^n$.
  If $X\in\boundary{p}$, then $\cK_X$, the
  kernel of $p(X)$, is non-zero.

  Given a smooth curve $X(t)\in\Smatng$,
  the derivative $G^{\prime}(t)$ of the
    function $G(t)=p(X(t))$ is equal to
\begin{equation}
 \label{eq:july7a6}
   G^{\prime}(t)=p^{\prime}(X(t))[X^{\prime}(t)].
\end{equation}
   The second derivative $G^{\prime\prime}(t)$
   is described in terms of the directional
    derivative and Hessian of $p$ by
\begin{equation}
  \label{eq:Gpp}
   G^{\prime\prime}(t)=p^{\prime\prime}
   (X(t))[X^\prime(t)]+p^\prime(X(t))[X^{\prime\prime}(t)],
\end{equation}
  an identity which leads to the following property
  of convex sets.

   \begin{lemma}[Lemma 2.1 of \cite{DHMind}]
    \label{lem:ph}
      Suppose $\posdom{p}$ is convex
     and let $(X,v)\in\Smatng \times \RR^n$ be given. If
     \begin{itemize}
      \item[(i)] $X\in\boundary{p}$;
      \item[(ii)]  $v\ne 0$ and  $p(X)v=0$; and
      \item[(iii)] $(-\delta, \delta)\ni t \mapsto X(t)$
      is a smooth curve in $\boundary{p}$ \\
         for which $X(0)=X$ and $p(X(t))v=0$,
     \end{itemize}
       then
     \begin{equation*}
      \langle p^{\prime\prime}(X(0))[X^\prime(0)]v,v\rangle\le 0\,.
     \end{equation*}
   \end{lemma}

Next   fix $X$ in $\boundary{p}$ and assume that $\cK_X$ is one-dimensional
   and spanned by $v$ and that $X$ is a smooth point (full rank point)
  of $\boundary{p}$. Under these hypothesis, 
   If  $H$ is in the clamped tangent plane,
 \begin{equation*}
  \mathcal T=\{H\in\Smatng : \ p^\prime(X)[H]v=0\},
 \end{equation*}
   to $\boundary{p}$ at $X$, then (by the implicit function
   theorem) there is a smooth curve $X(t)$
   in $\boundary{p}$ such that $X(0)=X$ and
   $X^\prime(0)=H$.
   This establishes the hypothesis of the lemma.
   Thus the clamped second fundamental form
   is negative semi-definite at $X$, for more details see Section 2 
   of  \cite{DHMind}.

  Next note that $\partial \mathcal P(p)$ is closed with respect
  to direct sums.
  In \cite{DHMind}  smoothness and irreducibility
  type conditions implied that  $\partial \mathcal P(p)$
  generically, but not universally, has positive curvature.
  Indeed, positive curvature could not be guaranteed  at points
  $(X,v)\in\partial \mathcal P(p)$ where the dimension of
  the kernel $p(X)$ exceeds one, which is the case
  for a direct sum of two points $(X,v),(Y,w)\in\partial \mathcal P(p)$.
  The trouble with higher dimensional kernels is that the
  implicit function theorem argument applied near such
  a point produces a curve which lies in $\cV(p)$, but
  perhaps not in the smaller set $\partial \mathcal P(p)$.

  On the other hand, under some additional fairly natural
  assumptions in \cite{DHMind}, convexity of $\mathcal P(p)$ implies
  positive curvature at  many points
  in $\partial \mathcal P(p)$.
This plus Corollary \ref{cor:sigmain} implies that if
$p$ is a
$0$-minimal degree
   defining polynomial for $\boundary{p}$, then $p$ has
   degree at most two.
This conclusion is stronger than the conclusion $p$ has
   degree at most four that was obtained in \cite{DHMind}.
   This sharpening of the bound
depends upon the stronger irreducibility
   hypothesis that is imposed here
and the extra mileage obtained from the introduction
and careful analysis of the relaxed Hessian. More precisely,
in the current terminology (see Subsection \ref{subsec:irreducible}),
the minimum degree hypothesis in \cite{DHMind} is that $p$ is a
1-minimal degree polynomial, whereas
Theorem \ref{thm:sigmain} in this paper assumes  that $p$ is a
0-minimal degree polynomial. Another recent result in this line that
is obtained without assuming  irreducibility states that if
$\posdom{p}$
  is both bounded and convex,
  then there is a positive integer  $\ell$ and a
 set of  real symmetric $\ell\times\ell$ matrices $A_1,\dots,A_\ell$
  such that $X\in\posdom{p}$ if and only if
\[
  I_\ell\otimes I_n -\sum A_j\otimes X_j \succ 0.
\]
 See \cite{HM} for details where it is proved with
 separating hyperplane techniques.


\begin{thebibliography}{99}

\bibitem[BGM06]{Ball06}
Joseph A. Ball,  Gilbert Groenewald and Tanit Malakorn,
Conservative structured non-commutative multidimensional linear systems, in:
{\it The state space method generalizations and applications},
 Oper. Theory Adv. Appl., {\bf 161}, Birkh\"{a}user, Basel, 2006, pp.
179--223,\\


\bibitem [BCR91] {BCR91}
Jacek Bochnak, Michel Coste and Marie-Francoise Roy, \ 
{\it Real Algebraic Geometry} \ Springer, Berlin 1998. \\


\bibitem[CHSY03]{CHSY03}
Juan F. Camino, J. William Helton, Robert E. Skelton and Jieping Ye,
Matrix
inequalities: A Symbolic Procedure to Determine Convexity
Automatically, Integral Equations Operator Theory,  {\bf 46} (2003), no. 4,
399-454. \\


\bibitem[DHM07a]{DHMjda}
Harry Dym, J. William Helton and Scott A. McCullough, The Hessian of a
non-commutative polynomial has numerous
negative eigenvalues, J. Anal. Math., {\bf 102}  (2007), 29--76. \\

\bibitem[DHM07b]{DHMind}
Harry Dym, J. William Helton and Scott A. McCullough,
Irreducible non-commutative defining polynomials
for convex sets have degree four or less,
Indiana Univ. Math. J.,  {\bf 56}  (2007),  no. 3, 1189--1231. \\


\bibitem[DGHM09]{DGHM} Harry Dym, Jeremy Greene, J. William  Helton and
Scott McCullough,  Classification of all
non-commutative polynomials whose Hessian has negative signature one and
a non-commutative second fundamental form, J. Anal. Math., {\bf 108} (2009), 
19--59 . \\

\bibitem[HM04]{HM04}
J. William Helton and Scott  McCullough,
Convex non-commutative polynomials have
degree two or less, SIAM Journal of Matrix
Analysis, {\bf 25} (2004), no. 4, 1124-1139. \\

\bibitem[HM]{HM}
 J. William Helton and Scott  McCullough,
 Every free basic semi-algebraic set has an LMI representation,
 preprint. \\

\bibitem[HdOSM]{HMS02}
J. William Helton, Mauricio C. de Oliveira,
Mark Stankus and Robert L.
Miller, {NCAlgebra Package for Mathematica}. \\ 


\bibitem[HMV06]{HMVjfa}
 J. William Helton, Scott  McCullough and Victor Vinnikov,
Non-commutative convexity arises from linear matrix
inequalities,  J. Funct. Anal.,  {\bf 240}  (2006),  no. 1, 105--191. \\


\bibitem[HMe98]{HMe98} J. William Helton and Orlando Merino, Sufficient
conditions
for the optimization of matrix functions, Proc. IEEE Conf. on Decision and
Control, 1998, pp. 1--5.\\

\bibitem[HP07]{HPpreprint}
J. William Helton and Mihai Putinar,  Positive Polynomials in Scalar and
Matrix Variables, the Spectral Theorem and Optimization, in:
 {\it Operator theory, structured matrices, and dilations},
Theta Ser. Adv. Math., {\bf 7}, Theta, Bucharest, 2007, pp. 229--306. \\


\bibitem[KVV06]{VVV06}
Dmitry S. Kalyuzhnyi-Verbovetski\u\i \, and Victor Vinnikov,
Non-commutative positive kernels and their matrix evaluations,
Proc. Amer. Math. Soc.  {\bf 134} (2006),  no. 3, 805--816.\\


\bibitem[Po06]{Po06}
Gelu Popescu, Operator theory on non-commutative varieties.
Indiana Univ. Math. J.,  {\bf 55}  (2006),  no. 2, 389--442.\\


\bibitem[SV06]{AIMfree06}
Dimitri Shlyakhtenko  and Dan Voiculescu,
Free analysis workshop summary: American Institute of Mathematics.
http://www.aimath.org/pastworkshops/freeanalysis.html, 2006.\\


\bibitem[Vo05]{Vo05}
Dan Voiculescu, Free probability and the von Neumann algebras of
free groups, Rep. Math. Phys. {\bf 55} (2005), no.1, 127-133.\\


\bibitem[Vo06]{Vo06}
Dan Voiculescu,  Symmetries arising from free probability theory,
in: Frontiers in number theory, physics, and geometry. {\bf I}
231-243 Springer, Berlin., 2006\\

\end{thebibliography}
\end{document}